\documentclass{ws-m3as}

\usepackage[utf8]{inputenc}
\usepackage{amssymb}
\usepackage{amsmath,bbm,bm}
\usepackage{amsfonts}
\usepackage{mathrsfs}
\usepackage{graphicx}
\usepackage{enumerate}
\usepackage{color}
\usepackage{cite}
\usepackage{soul}

\newtheorem{assumption}[theorem]{Assumption}

\DeclareMathAlphabet{\mathpzc}{OT1}{pzc}{m}{it}
\newcommand{\cEl}{\scalebox{1.27}{$\mathpzc{E}$}}
\newcommand{\cBl}{\scalebox{1.27}{$\mathpzc{B}$}}
\newcommand{\cAl}{\scalebox{1.27}{$\mathpzc{A}$}}
\newcommand{\R}{\mathbb{R}}

\DeclareMathOperator{\curl}{curl}
\DeclareMathOperator{\divg}{div}

\begin{document}

\markboth{R. Chill, T. Reis, T. Stykel}
{Analysis of a quasilinear coupled MQS model.}

%
\catchline{}{}{}{}{}
%

\title{Analysis of a quasilinear coupled magneto-quasistatic model: \\
solvability and regularity of solutions
}

\author{Ralph Chill
}

\address{Institut f\"ur Analysis, Technische Universit\"at Dresden \\
01062 Dresden, Germany \\
ralph.chill@tu-dresden.de}

\author{Timo Reis}

\address{Institut für Mathematik, Technische Universität Ilmenau, Weimarer Str. 32 \\
	98693 Ilmenau, Germany\\
timo.reis@tu-ilmenau.de}

\author{Tatjana Stykel}

\address{Institut f\"ur Mathematik \& Centre for Advanced Analytics and Predictive Sciences (CAAPS), Universit\"at Augsburg, Universit\"atsstr. 12a, 86159 Augsburg, Germany \\
stykel@math.uni-augsburg.de}

\maketitle

\begin{history}
\received{(Day Month Year)}
\revised{(Day Month Year)}
\comby{(xxxxxxxxxx)}
\end{history}

\begin{abstract}
We consider a~quasilinear model arising from dynamical magnetization. This model is described by a~magneto-quasistatic (MQS) approximation of Maxwell's equations. Assuming that the medium consists of a~conducting and a~non-conducting part, the derivative with respect to time is not fully entering, whence the system can be described by an abstract differential-algebraic equation. Furthermore, via magnetic induction, the system is coupled with an equation which contains the induced electrical currents along the associated voltages, which form the input of the system. The aim of this paper is to study well-posedness of the coupled MQS system and regularity of its solutions. Thereby, we rely on the classical theory of gradient systems on Hilbert spaces combined with the
concept of $\mathcal{E}$-subgradients using in particular the magnetic energy. The coupled MQS system precisely fits into this general framework.

\end{abstract}



\section{Introduction}\label{sec:intro}

Maxwell's equations play a~fundamental role in modeling and numerical analysis of electromagnetic field problems. They describe the dynamic and spatial behavior of the electromagnetic field in a~medium. These equations were discovered in the early 1860s and have since then received a~lot of attention by mathematicians, physicists and engineers \cite{Ja99}. The unknown variables are given by the $\R^3$-valued functions
\[\begin{aligned}
\bm{D}:&\text{ electric~displacement},&\bm{B}:& \text{ magnetic flux intensity},\\
\bm{E}:&\text{ electric~field intensity},&\bm{H}:&\text{ magnetic field intensity},\\
\bm{J}:&\text{ electric~current density},
\end{aligned}\]
which depend on a~spatial variable $\xi\in\mathit{\Omega}\subseteq\R^3$ and
time $t\in[0,T]\subset\mathbb{R}$. Assuming that there are no electric charges, Maxwell's equations are given by
\begin{align*}
\qquad\qquad\nabla\cdot \bm{D}&=0&&\qquad\qquad\text{(the medium contains no electric charges)},\\
\nabla\cdot \bm{B}&=0&&\qquad\qquad\text{(field lines of the magnetic~flux are closed)},\\
\nabla\times \bm{E}&=-{\textstyle\frac{\partial}{\partial t}}\bm{B}&&\qquad\qquad\text{(Faraday's law of induction)},\\
\nabla\times \bm{H}&=\bm{J}+{\textstyle\frac{\partial}{\partial t}}\bm{D} &&\qquad\qquad\text{(magnetic~flux law)},
\end{align*}
where $\nabla\cdot$ stands for the divergence and $\nabla\times$ denotes the curl of a~vector field.
In addition, the above variables fulfill {\em constitutive relations}, which are determined by the physical properties of the medium. Denoting the Euclidean norm by $\|\cdot\|_2$, the constitutive relations are, in the quasilinear and isotropic case, of the form
\begin{align*}
&\bm{D}(\xi,t)=\epsilon(\xi , \|\bm{E}(\xi,t)\|_2 )\bm{E}(\xi,t),\\
&\bm{H}(\xi,t)=\nu(\xi , \|\bm{B}(\xi,t)\|_2 ) \bm{B}(\xi,t),\\
&\bm{J}(\xi,t)=\sigma(\xi , \|\bm{E}(\xi,t)\|_2 ) \bm{E}(\xi,t)+\bm{J}_{\rm ext}(\xi,t)
\end{align*}
for some functions $\epsilon$, $\nu$, $\sigma:\mathit{\Omega}\times\R\to\R$ which respectively express the electric permittivity, magnetic reluctivity and electric conductivity of the material, and $\bm{J}_{\rm ext}$ stands for the externally injected currents.

In this paper, we consider a~problem where
the displacement currents $\tfrac{\partial}{\partial t}\bm{D}$ are negligible compared to the conduction currents, and therefore they can be omitted. We also assume that the conductivity is linear, that is, $\sigma(\xi):=\sigma(\xi , \|\bm{E}(\xi,t)\|_2 )$ does not depend on $\bm{E}$.
Further, under some additional topological conditions on $\mathit{\Omega}$, the fact that the magnetic flux intensity is divergence-free implies that we can make the ansatz $\bm{B}=\nabla\times\bm{A}$
for some function $\bm{A}$, which is called the {\em magnetic vector potential}.
Plugging this into Faraday's law of induction, we obtain
\mbox{$\nabla\times \bm{E}=-{\textstyle\frac{\partial}{\partial t}}\nabla\times\bm{A}$}
whence $\bm{A}$ can be chosen in a~way that
$\bm{E}=-{\textstyle\frac{\partial}{\partial t}}\bm{A}$.
Finally, inserting the constitutive relations for $\bm{H}$ and $\bm{J}$
into the magnetic flux law and using the derived representations for $\bm{B}$ and $\bm{E}$ in terms of $\bm{A}$,
 we obtain the so-called {\em magneto-quasistatic} (MQS) {\em approximation of Maxwell's equations} (also called {\em eddy current model}) given by
\begin{equation}
\tfrac{\partial}{\partial t}\left(\sigma\bm{A}\right) + \nabla \times \left(\nu(\cdot,\|\nabla \times \bm{A}\|_2)
\nabla \times \bm{A}\right)=\bm{J}_{\rm ext}\quad \text{ in } \mathit{\Omega}\times (0,T],
\label{eq:MQS01}
\end{equation}
see \cite{IdaBastos97,BdGCS18}.
Such equations are used, for example, in the modeling
of accelerator magnets, electric machines and transformers operating at low frequencies.
If a~part of the medium is non-conducting, then the function $\sigma$ vanishes on some subset of~$\mathit{\Omega}$. In this case, the MQS equation \eqref{eq:MQS01} becomes of degenerate parabolic or mixed parabolic-elliptic type.

The coupling of electromagnetic devices to an~external circuit can be realized as a~solid conductor model
or as a~stranded conductor model,
see \cite{SchoepsDGW13} for details. Here, we restrict ourselves to the stranded conductor model
where the external current is induced by $m$~windings
\begin{equation}
\bm{J}_{\rm ext}(\xi,t)=\chi(\xi) \bm{i}(t),\label{eq:MQS02}
\end{equation}
where $\bm{i}$ is the $\R^m$-valued current function, and $\chi$
is the $\mathbb{R}^{3\times m}$-valued winding density function which expresses the geometry of the windings. The windings are assumed to have an~internal resistance $R\in\R^{m\times m}$ to which $m$ time-dependent voltages are applied, where the latter is expressed by the $\R^m$-valued prescribed function $\bm{v}$. Further, by using the fact that the electric field induces another~voltage $\int_\mathit{\Omega} \chi^\top\bm{E} \, {\rm d}\xi$ along the windings, the relation
$\bm{E}=-{\textstyle\frac{\partial}{\partial t}}\bm{A}$
together with Kirchhoff's voltage law gives rise to
\begin{equation}
\tfrac{{\rm d}}{{\rm d} t} \int_\mathit{\Omega} \chi^\top\bm{A} \, {\rm d}\xi + R\, \bm{i} = \;\bm{v}\quad \text{ on } (0,T].
\label{eq:MQS03}
\end{equation}
Altogether, we obtain the quasilinear partial integro-differential-algebraic equations with unknown functions $\bm{A}$ and $\bm{i}$. These equations
are further equipped with some initial and boundary conditions which are specified in Section~\ref{sec:solution}.

Existence, uniqueness and regularity results for the linear MQS system \eqref{eq:MQS01} with the coupling relations \eqref{eq:MQS02} and \eqref{eq:MQS03} are - under some additional topological conditions on the conducting domain - presented in \cite{NicT14}. They are based on a~theo\-rem by Showalter on degenerate linear parabolic equations \cite{Show77}.
Quasilinear elliptic and non-degenerate parabolic equations in MQS field problems, that is, \eqref{eq:MQS01} and~\eqref{eq:MQS02} with a~prescribed function $\bm{i}$ and bounded and strictly positive mapping~$\sigma$,
have been studied in the context of optimal control in \cite{You13} and \cite{NicT17}, respectively.
The solvability of linear de\-ge\-ne\-ra\-te MQS equations have been investigated in \cite{ArnH12} by deriving a~unified variational formulation and in \cite{PauPTW21} by using the theory of evolutionary equations. Further,
a~comprehensive analysis of quasilinear MQS problems based on a~Schur complement approach has been provided
in \cite{BLS05}. However, the extension of these results to MQS systems with the coupling relation
has remained a~challenging problem which requires considerable care due to structurally different properties of the solution on conducting and non-conducting subdomains, and the additional integral constraint. In \cite{PauPic17}, a~convergence analysis of solutions of linear Maxwell equations to the MQS model is performed while taking the limit of the electric permittivity to zero.

In this paper, we follow a different approach
to analyze the existence and uniqueness of solutions of the general coupled MQS model  \eqref{eq:MQS01}--\eqref{eq:MQS03}. It relies on a~formulation of this model as an abstract differential-algebraic equation involving the subgradient of the magnetic energy. This novel approach gives rise not only to well-posedness but further allows us to prove regularity results for the solutions which are new even in the linear case. Hereby, the well-established theory of gradient systems involving subgradients of convex functions \cite{Bre73,Bar10} forms the basis for our generalization to the differential-algebraic case, which is subsequently applied to the coupled quasilinear MQS model.

The paper is organized as follows. Section~\ref{ssec:spaces} contains a~brief overview on the notation and function spaces used in the subsequent analysis. In Section~\ref{ssec:mqs}, we present the MQS model problem and state assumptions on geometry and material parameters. In Section~\ref{sec:solution}, we define the solution concept and prove the uniqueness result. In Section~\ref{sec:energy}, we introduce the magnetic energy and examine its essential properties. Section~\ref{sec:gradsys} and Section~\ref{sec:MQSsolv} contain our main results. First, we present some operator-theoretic results on a~class of abstract differential-algebraic systems involving subgradients. Thereafter, we show that the coupled nonlinear MQS system fits into this framework, which allows us to establish the existence and regularity properties of solutions to this model.

\section{Notations and Function Spaces}
\label{ssec:spaces}

Let $\mathbb{R}_{\ge0}$ denote the set of all nonne\-gative
real numbers, and $\R^{m\times n}$ the set of real matrices of size $m\times n$.
Further, $x\cdot y$ and $x\times y$ stand, respectively, for the Euclidean inner product and the cross product of $x,y\in\R^3$, and $\|x\|_2$ is the Euclidean norm of $x\in\R^3$. For a~function \mbox{$\bm{A}:\Omega\to\R^3$}, the expression $\|\bm{A}\|_2$ stands for the scalar-valued function $\xi\mapsto\|\bm{A}(\xi)\|_2$. The restriction of a~function $f$ to a~subset $S$ of its domain is denoted by~$\left.f\right|_S$.

The inner product on a Hilbert space $H$ is denoted by $\left\langle\cdot,\cdot\right\langle_H$, and the induced norm is denoted by $\|\cdot\|_H$.
The duality pairing between a Hilbert space $H$ and its dual space $H'$ is denoted by $\langle\cdot,\cdot\rangle$. Note that, throughout this paper, all spaces are assumed to be real.

The set of linear, bounded operators between two Banach spaces $X$ and $Y$ is denoted by $\mathcal{L}(X,Y)$ and, in the case $X=Y$, simply by $\mathcal{L}(X)$.
In the case when $X$ and $Y$ are Hilbert spaces, $A^*\in\mathcal{L}(Y,X)$ stands for the adjoint of $A$. Moreover,
${M^\top}\in\R^{n\times m}$ denotes the transpose of a matrix $M\in\R^{m\times n}$.

Lebesgue and first order Sobolev spaces of functions defined on a domain \mbox{$\mathit{\Omega}\subseteq\R^m$} and with values in a Banach space $X$ are denoted by $L^{p}(\mathit{\Omega};X)$, $W^{1,p}(\mathit{\Omega};X)$ and $H^{1}(\mathit{\Omega};X)$, respectively. We shortly write $L^{p}(\mathit{\Omega})$, $W^{1,p}(\mathit{\Omega})$ and $H^{1}(\mathit{\Omega})$ when $X=\R$. Especially, when $\mathit{\Omega} = \mathbb{I}$ is an interval, we also consider the space $L^p_{\rm loc} (\mathbb{I} ;X)$ which consists of all (equivalence classes of) functions $f:\mathbb{I}\to X$ such that $f\in L^{p}(\mathbb{K};X)$ for all compact intervals $\mathbb{K}\subseteq \mathbb{I}$. Similarly, one defines $W^{1,{p}}_{\rm loc}(\mathbb{I};X)$ and $H^{1}_{\rm loc}(\mathbb{I};X)$. The integrals of Banach space valued functions are to be understood in the Bochner sense. Writing $f\in C(\mathbb{I};X)$ for some measurable function $f:\mathbb{I}\to X$ means that there is a~representative in the equivalence class of $f$ which is continuous on $\mathbb{I}$. In this case, we use the notation $f(s+) := \lim\limits_{t\to s+} f(t)$, where the limit on the right-hand side is taken by using the continuous representative.

Let $\mathit{\Omega}\subseteq\R^3$ be an open domain. The weak (distributional) gradient of $\phi\in L^{{2}}(\mathit{\Omega})$ is denoted by $\nabla \phi$, and $\nabla\times \bm{A}$ stands for the weak curl of a vector field $\bm{A}\in L^{{2}}(\mathit{\Omega};\mathbb{R}^3)$. We consider the Sobolev space
\[\begin{aligned}
H(\curl,\mathit{\Omega}) & = \, \bigl\{\,\bm{A}\in L^2(\mathit{\Omega};\mathbb{R}^3)\enskip:\enskip\nabla\times\bm{A}\in L^2(\mathit{\Omega};\mathbb{R}^3)\,\bigr\},
\end{aligned}\]
which is a Hilbert space endowed with the inner product
$$
 \langle\bm{A},\bm{F} \rangle_{H(\curl,\mathit{\Omega})}  = \langle\bm{A},\bm{F} \rangle_{L^2(\mathit{\Omega};\mathbb{R}^3)}+
 \langle\nabla\times\bm{A},\nabla\times\bm{F} \rangle_{L^2(\mathit{\Omega};\mathbb{R}^3)}.
$$
If $\mathit{\Omega}\subset\R^3$ is bounded and has a~Lipschitz boundary $\partial\mathit{\Omega}$, then for almost any $\xi\in\partial\mathit{\Omega}$, there exists the outward unit normal vector $\bm{n}_o(\xi)\in\R^3$ of $\mathit{\Omega}$ in $\xi$. Here, ``almost any'' refers to the hypersurface Lebesgue measure in $\R^3$. It has been proven in \linebreak \cite[Theorem~I.2.11]{GiraRavi86} that any
$\bm{A}\in H(\curl,\mathit{\Omega})$ has a~well-defined tangential trace
$\bm{A}\times \bm{n}_o\in L^2(\partial\mathit{\Omega};\mathbb{R}^3)$.
This allows us to define the space
\[\begin{aligned}
H_0(\curl,\mathit{\Omega}) = &\, \bigl\{\,\bm{A}\in H(\curl,\mathit{\Omega})\enskip:\enskip \bm{A}\times \bm{n}_o = 0 \text{ on } \partial\mathit{\Omega} \,\bigr\},
\end{aligned}\]
which is a~closed subspace of {$H(\curl,\mathit{\Omega})$}.
It has also been proven in \cite[Theorem~I.2.11]{GiraRavi86} that for all $\bm{A}\in H_0(\curl,\mathit{\Omega})$ and $\bm{F}\in H(\curl,\mathit{\Omega})$,
\begin{equation}
\langle\nabla\times\bm{A},\bm{F} \rangle_{L^2(\mathit{\Omega};\mathbb{R}^3)}=\langle\bm{A},\nabla\times\bm{F}
\rangle_{L^2(\mathit{\Omega};\mathbb{R}^3)}.
\label{eq:curladj}
\end{equation}
This relation is an extension of the formula of integration by parts to the weak curl operator.

The space of di\-ver\-gen\-ce-\-free and square integrable functions is defined as
\begin{multline}
L^2(\divg\!=\!0,\mathit{\Omega};\mathbb{R}^3) \\
 = \bigl\{\bm{A}\in L^2(\mathit{\Omega};\R^3) \,\, :
\,\, \langle \bm{A},\nabla \psi\rangle_{L^2(\mathit{\Omega};\mathbb{R}^3)}=0
\text{ for all }\psi\in H^1_0 (\mathit{\Omega}) \,\bigr\} .
\label{eq:graddivorth}
\end{multline}
{It} is a~closed subspace of $L^{2}(\mathit{\Omega};\mathbb{R}^3)$ and, therefore, a~Hilbert space with respect to the standard inner product in
$L^{2}(\mathit{\Omega};\mathbb{R}^3)$. Recall here, that as usual the Sobolev space $H^1_0 (\Omega )$ is the closure of the space of test functions in $H^1 (\Omega )$.

\section{Model Problem and Assumptions}
\label{ssec:mqs}

In this section, we consider the coupled MQS system as motivated in Section~\ref{sec:intro} in more detail.
 We start with the introduction of the model, and, thereafter, we collect the assumptions on the spatial domain and the system parameters.

\subsection{The coupled MQS model}
Let $\mathit{\Omega}\subset \mathbb{R}^3$ be a~bounded domain
 with boundary $\partial \mathit{\Omega}$ and let $T>0$. We consider the~coupled MQS system in magnetic vector potential formulation
\begin{subequations}\label{eq:MQS}
\begin{align}
\tfrac{\partial}{{\partial}t}\left(\sigma\bm{A}\right) + \nabla \times \left(\nu(\cdot,\|\nabla \times \bm{A}\|_2)
\nabla \times \bm{A}\right)  = & \; \chi\,\bm{i} & \text{ in } &\mathit{\Omega}\times (0,T], \label{eq:MQS1} \\
\tfrac{{\rm d}}{{\rm d} t} \int_\mathit{\Omega} \chi^\top\bm{A} \, {\rm d}\xi + R\, \bm{i} = &\;\bm{v}\label{eq:MQScoupl} & \text{ on } &(0,T], \\
\bm{A}\times \bm{n}_o  = &\; 0 & \mbox{in }& \partial \mathit{\Omega}\times (0,T],
\label{eq:MQSbc}\\[2mm]
\sigma\bm{A}(\cdot,0) = &\; \sigma\bm{A}_0 &\text{ in }&\mathit{\Omega},
\label{eq:MQSic1}\\
\int_\mathit{\Omega} \chi^\top\bm{A}(\cdot,0) \, {\rm d}\xi=&\, \int_\mathit{\Omega} \chi^\top\bm{A}_0\,
{\rm d}\xi,&& \label{eq:MQSic2}
\end{align}
\end{subequations}
where $\bm{A}:\mathit{\Omega} \times [0,T]\to\mathbb{R}^3$ is the magnetic vector potential,
$\nu:\mathit{\Omega}\times\mathbb{R}_{\ge0}\to\mathbb{R}_{\ge0}$ is the magnetic reluctivity,
$\sigma:\mathit{\Omega}\to\mathbb{R}_{\ge0}$ is the electric conductivity,
$\bm{v}:[0,T]\to\mathbb{R}^m$ and $\bm{i}:[0,T]\to\mathbb{R}^m$  are, respectively, the voltage and the electrical current through the electromagnetic conductive contacts. Furthermore,
$\chi:\mathit{\Omega}\to\mathbb{R}^{3\times m}$ is the winding function,
$R\in\mathbb{R}^{m\times m}$ is the resistance of the winding,
and $\bm{A}_0: \mathit{\Omega}\to\mathbb{R}^3$
is the initial value for the magnetic vector potential.
The boundary condition \eqref{eq:MQSbc} implies that the magnetic flux through the boundary $\partial \mathit{\Omega}$ is zero. Moreover, equations
 \eqref{eq:MQSic1} and \eqref{eq:MQSic2} describe the initial conditions for the magnetic vector potential.
Note that we only initialize the parts of $\bm{A}$ whose derivatives occur in \eqref{eq:MQS1} and \eqref{eq:MQScoupl}.
The coupled MQS system \eqref{eq:MQS} can be considered as a~control system, where the voltage~$\bm{v}$ takes the role of the input and $(\bm{A},\bm{i})$ is the state.

\subsection{The spatial domain}
This subsection contains the assumptions on the spatial domain $\mathit{\Omega}$ which are made throughout this paper.

\begin{assumption}[Spatial domain, geometry and topology]
\label{ass:omega}
{\em The set \mbox{$\mathit{\Omega}\subset\mathbb{R}^3$} is a~simply connected bounded Lipschitz domain, which is decomposed into two Lipschitz regular, open subsets \mbox{$\mathit{\Omega}_{C}$, $\mathit{\Omega}_{I}\subset \mathit{\Omega}$}, called, respectively, {\em conducting} and {\em non-conducting subdomains}, such that
$\overline{\mathit{\Omega}}_{C}\subset {\mathit{\Omega}}$ and
$\mathit{\Omega}_{I}=\mathit{\Omega}\setminus \overline{\mathit{\Omega}}_C$. Furthermore, the subdomain $\mathit{\Omega}_C$ is connected, and $\mathit{\Omega}_{I}$ has finitely many connected components $\mathit{\Omega}_{I,1},\ldots,\mathit{\Omega}_{I,q}$ and $\mathit{\Omega}_{I,{\rm ext}}$, where \vspace*{-1mm}
\begin{romanlist}[a)]
\item each of the sets $\mathit{\Omega}_{I,1},\ldots,\mathit{\Omega}_{I,q}$ has exactly one boundary component, these are denoted by $\Gamma_1,\ldots,\Gamma_q$ and called the {\em internal interfaces};
\item the {\em external non-conducting subdomain} $\mathit{\Omega}_{I,{\rm ext}}$ has two boundary components $\partial\mathit{\Omega}$ and the {\em external interface} $\Gamma_{{\rm ext}}:=\overline{\mathit{\Omega}}_{I,{\rm ext}}\cap \overline{\mathit{\Omega}}_{C}$.
\end{romanlist}
}
\end{assumption}

Note that by a~{\em boundary component} of a~subdomain $\mathit{\Omega}_*\subseteq\mathit{\Omega}$, we mean a~connected component of its boundary $\partial\mathit{\Omega}_*$.
Since $\mathit{\Omega}_{I}$ has a~Lipschitz boundary, the closed connected components
$\overline{\mathit{\Omega}}_{I,1},\ldots,\overline{\mathit{\Omega}}_{I,q}$ and $\overline{\mathit{\Omega}}_{I,{\rm ext}}$ are disjoint.
The subdomains $\mathit{\Omega}_{I,1},\ldots,\mathit{\Omega}_{I,q}$ can be interpreted as ``interior cavities'' of the conducting subdomain $\mathit{\Omega}_C$. In particular,  we do not assume that the conducting subdomain  is simply connected, it may also have some ``handles''. Note that the lapidarily formulated notions of ``interior cavities'' and ``handles'' can be made mathematically precise in terms of the so-called {\em Betti numbers}~\cite{Ros19}.
We later assume that the electric conductivity is a~scalar multiple of the indicator function on the conducting subdomain, which justifies the above naming.

\subsection{The space $X(\mathit{\Omega},\mathit{\Omega}_C)$, the initial conditions and the winding function}

Next, we present a~space in which the solutions of the coupled MQS system \eqref{eq:MQS} evolve. As a~preliminary thought, note that equation~\eqref{eq:MQS1} does not change, if we replace $\bm{A}$ by $\bm{A} + \nabla \psi$ for an arbitrary but fixed $\psi \in H^1(\mathit{\Omega})$ which is constant on each boundary component $\Gamma_1,\ldots,\Gamma_q,\Gamma_{\rm ext}$ and $\partial\mathit{\Omega}$.
Therefore, we restrict our considerations to solutions which are pointwise orthogonal to all gradient fields of functions being constant on each set $\Gamma_1,\ldots,\Gamma_q,\Gamma_{\rm ext}$ and $\partial\mathit{\Omega}$. Since the conducting and non-conducting subdomains are both Lipschitz regular, the trace theorem \cite[Theorem~1.39]{Yagi10} yields that the following space is well-defined:
\begin{align}
\label{eq:statespaceorth}
G(\mathit{\Omega},\mathit{\Omega}_C)
=\Bigl\{ \nabla \psi \Bigr. \enskip:\enskip
    &\psi\in H^1_0(\mathit{\Omega})\text{ s.t.\ } \exists\, c_1,\ldots,c_q,c_{\rm ext}\in\R\text{ with }\\
    &\!\left.\psi\right|_{\Gamma_i} \equiv c_i \text{ for }
		i=1,\ldots,q,\;  \Bigl. \left.\psi\right|_{\Gamma_{\rm ext}}\! \equiv c_{\rm ext} \Bigr\}. \!\!\!\!\!\!\nonumber
\end{align}
We are seeking for solutions with values in the orthogonal space of $G(\mathit{\Omega},\mathit{\Omega}_C)$ in $L^2(\mathit{\Omega};\mathbb{R}^3)$, that is, in the space
\begin{equation}
X(\mathit{\Omega},\mathit{\Omega}_C)=\left\{\bm{A} \in L^2(\mathit{\Omega};\mathbb{R}^3)\enskip:\enskip
\langle\bm{A}, \bm{F} \rangle_{L^2(\mathit{\Omega};\mathbb{R}^3)} = 0 \;\text{ for all } \bm{F}\in G(\mathit{\Omega},\mathit{\Omega}_C)\right\},
\label{eq:statespace}
\end{equation}
which is a~Hilbert space when it is equipped with the inner product $\langle \cdot , \cdot \rangle_{L^2(\mathit{\Omega};\mathbb{R}^3)}$.
We further consider the space
\begin{equation}
X_0(\curl,\mathit{\Omega},\mathit{\Omega}_C)=H_0(\curl,\mathit{\Omega})\cap X(\mathit{\Omega},\mathit{\Omega}_C),
\label{eq:statespace2}
\end{equation}
which is again a Hilbert space, now provided with the inner product in $H_0(\curl,\mathit{\Omega})$.
The space $X(\mathit{\Omega},\mathit{\Omega}_C)$ enables us to formulate our assumption on the initial magnetic vector potential and the winding function.

\begin{assumption}[Initial magnetic vector potential and winding function]\label{ass:init}\!\!
\vspace*{-6mm}
{\em \begin{romanlist}[a)]
\item\label{ass:initial}
The initial magnetic vector potential $\bm{A}_0:\mathit{\Omega}\to\mathbb{R}^3$ belongs to
$X(\mathit{\Omega},\mathit{\Omega}_C)$.
\item\label{ass:winding}
The columns of the winding function $\chi:\mathit{\Omega}\to\mathbb{R}^{3\times m}$,
denoted by $\chi_1,\ldots,\chi_m$,
belong to $X(\mathit{\Omega},\mathit{\Omega}_C)$.
\end{romanlist}
}
\end{assumption}
Note that
\[
 \bigl\{ \nabla\psi\; :\; \psi \in H^1_0(\mathit{\Omega}_C \cup \mathit{\Omega}_I ) \bigr\}
\subseteq G(\mathit{\Omega},\mathit{\Omega}_C) \subseteq \bigl\{ \nabla\psi\; :\; \psi\in H^1_0(\mathit{\Omega} ) \bigr\} ,
\]
Therefore, by using \eqref{eq:graddivorth}, we obtain
\begin{equation}\label{eq:divfree}
 L^2(\divg\!=\!0, \mathit{\Omega};\R^3) \subseteq X(\mathit{\Omega},\mathit{\Omega}_C) \subseteq
L^2(\divg\!=\!0, \mathit{\Omega}_C \cup \mathit{\Omega}_I ;\R^3) .
\end{equation}
In particular, the first inclusion in \eqref{eq:divfree} implies that
Assumption~\ref{ass:init}~\ref{ass:winding}) on the winding function $\chi$ is fulfilled, if the columns of~$\chi$ belong to $L^2(\divg\!=\!0, \mathit{\Omega};\R^3)$.
In practice, the current is  often injected through the contacts in the non-conducting subdomain $\mathit{\Omega}_I$, that is, $\mbox{supp}(\chi) \subset\mathit{\Omega}_I$. In this case, $\chi_1,\ldots,\chi_m\in X(\mathit{\Omega},\mathit{\Omega}_C)$ is even equivalent to
$\chi_1,\ldots,\chi_m\in L^2(\divg\!=\!0,\mathit{\Omega};\R^3)$.

Further, note that any $\bm{A}\in X(\mathit{\Omega},\mathit{\Omega}_C)$ is indeed divergence-free on the conducting subdomain $\mathit{\Omega}_C$ as well as on the non-conducting subdomain $\mathit{\Omega}_I$. Since the curl of a~function is divergence-free, this yields
\begin{equation*}
\nabla\times \bm{A}\in X(\mathit{\Omega},\mathit{\Omega}_C)\qquad \text{for all }
	\bm{A}\in H_0(\curl,\mathit{\Omega}).
\label{eq:curlsubset}
\end{equation*}
In the following, we collect some further properties of the space $X(\mathit{\Omega},\mathit{\Omega}_C)$.
The subsequent lemma establishes that this space is closed with respect to multiplication by the indicator function of the conducting subdomain $\mathit{\Omega}_C$.

\begin{lemma}\label{lem:sigmamult}
Let $\mathit{\Omega}\subset\mathbb{R}^3$ and $\mathit{\Omega}_C$, $\mathit{\Omega}_I\subseteq\mathit{\Omega}$ be as in Assumption~\textup{\ref{ass:omega}}, and let $\mathbbm{1}_{\mathit{\Omega}_C}:\mathit{\Omega}\to\R$ be the indicator function of the set $\mathit{\Omega}_C$, that is,
\begin{equation*}\label{eq:indfun}
\mathbbm{1}_{\mathit{\Omega}_C}(\xi)=\begin{cases}1, &\,\xi\in\mathit{\Omega}_C,\\0,&\,\xi\notin\mathit{\Omega}_C .\end{cases}
\end{equation*}
Then, for every $\bm{A}\in X(\mathit{\Omega},\mathit{\Omega}_C)$, one has $\mathbbm{1}_{\mathit{\Omega}_C} \bm{A}\in X(\mathit{\Omega},\mathit{\Omega}_C)$.
\end{lemma}

\begin{proof}
Let $\bm{A}\in X(\mathit{\Omega},\mathit{\Omega}_C)$ and let $\bm{F}\in G(\mathit{\Omega},\mathit{\Omega}_C)$ be arbitrary. Then there exists some $\psi\in H^1_0(\mathit{\Omega})$ and $c_1,\ldots,c_q,c_{\rm ext}\in\R$ such that $\bm{F}=\nabla \psi$, $\left.\psi\right|_{\Gamma_i} \equiv c_i$ for \mbox{$i=1,\ldots,q$}, and $\left.\psi\right|_{\Gamma_{\rm ext}} \equiv c_{\rm ext}$.
Consider the function
\[
\widetilde{\psi}(\xi)=\begin{cases}\psi(\xi)-c_{\rm ext},&\,\xi\in\mathit{\Omega}_C,\\c_i-c_{\rm ext},&\,\xi\in\mathit{\Omega}_{I,i},\; i=1,\ldots,q,\\0,&\,\xi\in\mathit{\Omega}_{I,{\rm ext}}.\end{cases}
\]
Then  $\widetilde{\psi}\in H^1(\mathit{\Omega}_C \cup \mathit{\Omega}_I)$ and the traces of $\widetilde{\psi}$ from both sides of the interfaces $\Gamma_1,\ldots,\Gamma_q$ and $\Gamma_{{\rm ext}}$ coincide. Hence, $\widetilde{\psi}\in H^1(\mathit{\Omega})$ and, by $\widetilde{\psi}|_{\mathit{\Omega}_{I,{\rm ext}}}=0$, we even have $\widetilde{\psi}\in H^1_0(\mathit{\Omega})$. This gives $\nabla \widetilde{\psi}\in G(\mathit{\Omega},\mathit{\Omega}_C)$. Taking into account that $\nabla\widetilde{\psi}$ vanishes on $\mathit{\Omega}_I$, we obtain
$$
0 =\int_{\mathit{\Omega}} \bm{A} \cdot \nabla \widetilde{\psi}\, {\rm d}\xi
  =\int_{\mathit{\Omega}_C} \bm{A} \cdot \nabla \widetilde{\psi}\,  {\rm d}\xi
	=\int_{\mathit{\Omega}_C} \bm{A}\cdot \nabla {\psi}\,  {\rm d}\xi
	=\int_{\mathit{\Omega}} (\mathbbm{1}_{\mathit{\Omega}_C}\bm{A}) \cdot \bm{F}\,  {\rm d}\xi,
$$
and, thus, $\mathbbm{1}_{\mathit{\Omega}_C}\bm{A}\in X(\mathit{\Omega},\mathit{\Omega}_C)$.
\end{proof}

Next, we show that $X_0(\curl,\mathit{\Omega},\mathit{\Omega}_C)$ is dense in $X(\mathit{\Omega},\mathit{\Omega}_C)$. Moreover, we derive an~estimate on the $L^2$-norm of $\bm{A}\in X_0(\curl,\mathit{\Omega},\mathit{\Omega}_C)$ by means of the $L^2$-norm of $\nabla\times \bm{A}$ and the $L^2$-norm of the restriction of $\bm{A}$ to the conducting subdomain $\mathit{\Omega}_C$.

\newpage
\begin{lemma}\label{lem:denscoerc}
Let $\mathit{\Omega}\subset\mathbb{R}^3$ and $\mathit{\Omega}_C$, $\mathit{\Omega}_I\subseteq\mathit{\Omega}$ be as in Assumption~\textup{\ref{ass:omega}}. Then $X_0(\curl,\mathit{\Omega},\mathit{\Omega}_C)$ is dense in $X(\mathit{\Omega},\mathit{\Omega}_C)$. Further, there exists $L_C>0$ such that, for all $\bm{A}\in X_0(\curl,\mathit{\Omega},\mathit{\Omega}_C)$,
\begin{equation}
\|\bm{A}\|_{L^2(\mathit{\Omega};\R^3)}^2\leq L_C\,\left(\|\bm{A}\|_{L^2(\mathit{\Omega}_C;\R^3)}^2+\|\nabla\times \bm{A}\|_{L^2(\mathit{\Omega};\R^3)}^2\right) .
\label{eq:curlest}
\end{equation}
\end{lemma}

\begin{proof}
The existence of a~constant $L_C>0$ such that the estimate \eqref{eq:curlest} holds for all \mbox{$\bm{A}\in X_0(\curl,\mathit{\Omega},\mathit{\Omega}_C)$} immediately follows from \cite[Lemma~4]{BLS05}.
Hence, it only remains to prove the density statement.
To this end, let $\bm{A}\in X(\mathit{\Omega},\mathit{\Omega}_C)$ and $\varepsilon >0$. {Since $H_0(\curl,\mathit{\Omega})$ contains the space of test functions on $\mathit{\Omega}$, it is dense in $L^2(\mathit{\Omega};\mathbb{R}^3)$. Hence,} there exists some $\bm{C}\in H_0(\curl,\mathit{\Omega})$ such that
\[
\|\bm{A}-\bm{C}\|_{L^2(\mathit{\Omega};\mathbb{R}^3)}<\varepsilon.
\]
Now consider the orthogonal decomposition $\bm{C}=\bm{G}+\nabla \psi$
with $\bm{G}\in X(\mathit{\Omega},\mathit{\Omega}_C)$
and $\nabla\psi\in G(\mathit{\Omega},\mathit{\Omega}_C)$.
Since
$\nabla\times(\nabla \psi)=0$, we have $\bm{G}=\bm{C}-\nabla {\psi}\in H(\curl,\mathit{\Omega})$. Further, since the boundary trace of $\psi\in H^1_0(\mathit{\Omega})$ is constant, the tangential component of the gradient of $\psi$ vanishes at $\partial\mathit{\Omega}$, that is,
 $\nabla\psi\times \bm{n}_o=0$ on $\partial\mathit{\Omega}$. This gives rise to $\bm{G}\in H_0(\curl,\mathit{\Omega})$. Then, by using the Pythagorean identity, we obtain
\[\begin{aligned}
\|\bm{A}-\bm{G}\|_{L^2(\mathit{\Omega};\mathbb{R}^3)}^2\leq\,& \|\bm{A}-\bm{G}\|_{L^2(\mathit{\Omega};\mathbb{R}^3)}^2+\|\nabla \psi\|_{L^2(\mathit{\Omega};\mathbb{R}^3)}^2\\
=\,&\|\bm{A}-(\bm{G}+\nabla \psi)\|_{L^2(\mathit{\Omega};\mathbb{R}^3)}^2=\|\bm{A}-\bm{C}\|_{L^2(\mathit{\Omega};\mathbb{R}^3)}^2<\varepsilon^2.
\end{aligned}\]
This completes the proof.
\end{proof}

{\subsection{The material parameters}
We now state the assumptions on the magnetic reluctivity, the electric conductivity and the resistance matrix.} Note that we consider only isotropic materials without hysteresis effects.

\begin{assumption}[Material parameters]\!\!\label{ass:material}
	\vspace*{-1mm}
{\em \begin{romanlist}[a)]
\item \label{ass:material1}
The electric conductivity $\sigma:\mathit{\Omega}\to\mathbb{R}_{\geq 0}$ is of the form
$\sigma=\sigma_C\mathbbm{1}_{\mathit{\Omega}_C}$ with  a~real number $\sigma_C>0$.

\item \label{ass:material2}
The magnetic reluctivity $\nu:\mathit{\Omega}\times \mathbb{R}_{\ge0}\to\mathbb{R}_{\ge0}$ has the following properties: \vspace*{-1mm}
\begin{romanlist}[(i)]
\item\label{ass:material2a} $\nu$ is measurable;
\item\label{ass:material2c} the function $\zeta\mapsto\nu(\xi,\zeta)\zeta$ is strongly monotone with a~monotonicity constant \mbox{$m_{\nu}>0$} independent of $\xi\in\mathit{\Omega}$. In other words, there exists $m_{\nu}>0$ such that
\[
  \bigl(\nu(\xi,\zeta) \zeta-\nu(\xi,\varsigma)\varsigma\bigr)(\zeta-\varsigma)\geq m_{\nu} (\zeta-\varsigma)^2 \enskip\text{ for all } \xi\in\mathit{\Omega}, \enskip \zeta,\varsigma\in\mathbb{R}_{\ge0};
\]

\item\label{ass:material2d}
the function $\zeta\mapsto\nu(\xi,\zeta)\zeta$ is Lipschitz continuous with a~Lipschitz constant $L_{\nu}>0$ independent of $\xi\in\mathit{\Omega}$. In other words, there exists $L_{\nu}>0$ such that
\[
  |\nu(\xi,\zeta)\zeta-\nu(\xi,\varsigma)\varsigma| \leq L_{\nu} |\zeta-\varsigma|\quad\text{ for all } \xi\in\mathit{\Omega}, \enskip\zeta,\varsigma\in\mathbb{R}_{\ge0}.
\]
\end{romanlist}

\item\label{ass:resistance} The resistance matrix $R\in\mathbb{R}^{m\times m}$ is symmetric and positive definite.
\end{romanlist}
}
\end{assumption}

It immediately follows from the second and third conditions on the magnetic reluctivity $\nu$
that
\begin{equation}\label{eq:Mnugreater}
m_\nu \leq \nu(\xi,\zeta)\leq L_{\nu} \quad \text{ for all } \xi\in \mathit{\Omega} \text{ and all }
\zeta \,>0.
\end{equation}

\section{The Solution Concept}
\label{sec:solution}

In this section, we explain what we mean by a~solution of the coupled MQS system~\eqref{eq:MQS}
and prove the uniqueness result. Notice that by using the canonical isomorphism $L^1 (\mathit{\Omega}\times [0,T]) = L^1 ([0,T] ; L^1 (\mathit{\Omega} ))$, we identify integrable functions defined on $\mathit{\Omega}\times[0,T]$ with integrable functions $[0,T]\to L^1 (\mathit{\Omega} )$. Sometimes we skip the placeholders for the arguments for sake of brevity.

\begin{definition}[Solution of the MQS system]\label{def:sol}
Let $\mathit{\Omega}\subset\mathbb{R}^3$ with subdomains $\mathit{\Omega}_C$ and $\mathit{\Omega}_I$ satisfy Assumption~\textup{\ref{ass:omega}}, and let $X(\mathit{\Omega},\mathit{\Omega}_C)$ and $X_0(\curl\mathit{\Omega},\mathit{\Omega}_C)$ be defined as in \eqref{eq:statespace} and \eqref{eq:statespace2}, respectively. Further, let the initial and winding functions be as in Assumption~\textup{\ref{ass:init}} and the material parameters as in Assumption~\textup{\ref{ass:material}}. Let $T>0$ be fixed and $\bm{v}\in L^2([0,T];\mathbb{R}^m)$. Then $(\bm{A},\bm{i})$ with $\bm{A}:\overline{\mathit{\Omega}}\times [0, T]\to\mathbb{R}^3$ and $\bm{i}:[0,T]\to\mathbb{R}^m$ is called a~{\em weak solution}
of the coupled MQS system \eqref{eq:MQS}, if\vspace*{-1mm}
\begin{romanlist}[a)]
\item\label{item:sol1} $\sigma\bm{A} \in C([0,T]; X(\mathit{\Omega},\mathit{\Omega}_C))\cap H_{\rm loc}^1((0,T];X(\mathit{\Omega},\mathit{\Omega}_C))$ and $\sigma\bm{A}(0)=\sigma\bm{A}_0$,

\vspace*{.5mm}
\item\label{item:sol2} $\int_\mathit{\Omega} \chi^\top\bm{A} \, {\rm d}\xi \in C([0,T] ; \R^m) \cap H_{\rm loc}^1((0,T];\mathbb{R}^m)$ and $\int_\mathit{\Omega} \chi^\top\bm{A}(0) \, {\rm d}\xi=\int_\mathit{\Omega} \chi^\top\bm{A}_0 \, {\rm d}\xi$,

\vspace*{.5mm}
\item\label{item:sol3} $\bm{A}\in L^2([0,T];X_0(\curl,\mathit{\Omega},\mathit{\Omega}_C))$ and $\bm{i}\in  L_{\rm loc}^2((0,T];\mathbb{R}^m)$,

\vspace*{.5mm}
\item\label{item:sol7} for all $\bm{F}\in X_0(\curl,\mathit{\Omega},\mathit{\Omega}_C)$, the equations
\begin{equation}
\arraycolsep=2pt
\begin{array}{rcl}
\displaystyle{\tfrac{\rm d }{{\rm d} t}\!\int_\mathit{\Omega} \!\sigma \bm{A}(t)\cdot \bm{F}\, {\rm d}\xi
+\!\!	\int_\mathit{\Omega} \!\nu(\cdot,\|\nabla\!\times\! \bm{A}(t)\|_2)(\nabla\!\times\! \bm{A}(t))
\cdot(\nabla\!\times\! \bm{F})\, {\rm d}\xi} & \!= &\!\!
\displaystyle{\int_\mathit{\Omega} \!\chi\, \bm{i}(t)\cdot \bm{F}\,  {\rm d}\xi,} \\[2mm]
\displaystyle{\tfrac{\rm d }{{\rm d} t}\!\int_\mathit{\Omega} \chi^\top\bm{A}(t)\, {\rm d}\xi+R\,\bm{i}(t) }& = &\! \bm{v}(t)
\end{array}
\label{eq:weak}
\end{equation}
hold for almost all $t\in[0,T]$.
\end{romanlist}
\end{definition}

\begin{remark}\label{rem:infdimds}
The first equation in \eqref{eq:weak} is motivated by an integration by parts with the weak curl operator \textup{(}cf.\ \eqref{eq:curladj}\textup{)}.  In particular, if $(\bm{A},\bm{i})$ is a~classical solution in the sense that all partial derivatives in  \eqref{eq:MQS} exist in the classical sense and are continuous up to the boundary of $\mathit{\Omega}$, and \eqref{eq:MQS}
holds pointwise everywhere in $\mathit{\Omega}\times [0,T]$ together with the boundary condition, then $(\bm{A},\bm{i})$ is a~weak solution.
\end{remark}

Condition \ref{item:sol1}) in Definition~\ref{def:sol} means that
$\tfrac{d}{dt}\sigma\bm{A}:[0,T]\to X_0(\curl,\mathit{\Omega},\mathit{\Omega}_C)'$ is measurable, where $X_0(\curl,\mathit{\Omega},\mathit{\Omega}_C)'$ is the dual of $X_0(\curl,\mathit{\Omega},\mathit{\Omega}_C)$ with respect to the pivot space $X(\mathit{\Omega},\mathit{\Omega}_C)$. Define the operators
\allowdisplaybreaks
\begin{subequations}\label{eq:EAiiop}
\begin{align}
\cEl_{11}:\!\!\!\!&& X(\mathit{\Omega},\mathit{\Omega}_C)&\,\to\,X(\mathit{\Omega},\mathit{\Omega}_C),\label{eq:E11} \\
          &&                               \bm{A}&\;\mapsto \,  \sqrt{\sigma} \bm{A} ,\nonumber\\[2mm]
\cEl_{21}:\!\!\!\!&& X(\mathit{\Omega},\mathit{\Omega}_C)&\,\to\,\mathbb{R}^m,\label{eq:E21} \\
          &&    \bm{A}&\,\mapsto \displaystyle{\,\int_\mathit{\Omega} \chi^\top\bm{A} \, {\rm d}\xi,}
					\nonumber\\[2mm]
\cAl_{11}:\!\!\!\!&&\hspace*{-4mm}X_0(\curl,\mathit{\Omega},\mathit{\Omega}_C) &\,\to \,X_0(\curl,\mathit{\Omega},\mathit{\Omega}_C)',\label{eq:A11}\\
          &&\bm{A}&\,\mapsto\,\displaystyle{\biggl(\bm{F}\mapsto
	\int_\mathit{\Omega} \nu(\cdot,\|\nabla\times \bm{A}\|_2)(\nabla\times \bm{A})\cdot(\nabla\times \bm{F})\, {\rm d}\xi\biggr),}\hspace*{-10mm}\nonumber \\[2mm]
\cAl_{12}:\!\!\!\!&&\mathbb{R}^m                  &\,\to\, X(\mathit{\Omega},\mathit{\Omega}_C),\label{eq:A12}\\
          &&\bm{i}&\,\mapsto\,  \chi\, \bm{i},\nonumber \\[2mm]
\cAl_{22}:\!\!\!\!&&\mathbb{R}^m                        &\,\to\,\mathbb{R}^m,\label{eq:A22} \\
          &&\bm{i}&\,\mapsto\, R^{1/2} \,\bm{i}, \nonumber
\end{align}
\end{subequations}
and
$$
\arraycolsep=2pt
\begin{array}{rrcl}
\qquad \cEl:&X(\mathit{\Omega},\mathit{\Omega}_C)\times \mathbb{R}^m\,&\to& X(\mathit{\Omega},\mathit{\Omega}_C) \times \mathbb{R}^m,\qquad\qquad\qquad\qquad\qquad\qquad\\
&(\bm{A},\bm{i})\,&\mapsto&(\cEl_{11}^* \cEl_{11}^{}\bm{A},\cEl_{21}^{}\bm{A}),\\[2.5mm]
\qquad \cAl: & X_0(\curl,\mathit{\Omega},\mathit{\Omega}_C) \times \mathbb{R}^m\, & \to & X_0(\curl,\mathit{\Omega},\mathit{\Omega}_C)'\times \mathbb{R}^m,\\
&(\bm{A},\bm{i})\,&\mapsto&(-\cAl_{11}^{}(\bm{A}) +\cAl_{12}^{}\,\bm{i},-\cAl_{22}^*\cAl_{22}^{}\,\bm{i}),\\[2.5mm]
\qquad \cBl:& \mathbb{R}^m\,&\to & X_0(\curl,\mathit{\Omega},\mathit{\Omega}_C)' \times \mathbb{R}^m,\\
&\bm{v}\,&\mapsto&(0,\bm{v}).
\end{array}
$$
Then the coupled MQS system \eqref{eq:MQS} can equivalently be written as an~abstract differential-algebraic system
  \begin{equation}\label{eq:DS1}
\tfrac{\rm d}{{\rm d}t}\cEl x(t) \,= \cAl (x(t))\,  +\, \cBl  u(t),\qquad \cEl x(0)=\cEl x_0,
\end{equation}
with the input $u(t)=\bm{v}(t)$, the state $x(t)=(\bm{A}(t), \bm{i}(t))$, and the initial condition
$x_0=(\bm{A}_0, 0)$.
Note that the operators $\cEl_{11}$, $\cEl_{21}$ (and thus also $\cEl$), $\cAl_{12}$ and $\cAl_{22}$ are linear, whereas $\cAl_{11}$ (and thus also $\cAl$) is nonlinear unless the reluctivity $\nu$ is constant with respect to the second argument.

Our aim is to derive existence, uniqueness and qualitative behavior of solutions of the MQS system \eqref{eq:MQS}.
The existence proof is more involved and is subject of Section~\ref{sec:gradsys}. The proof of uniqueness is by far more simple and is presented here in Theorem \ref{thm:uniqueness} below. The essential ingredient is that the operator $\cAl_{11}$ is monotone in some sense. This is subject of the subsequent lemma, which is a~straightforward consequence of Assumption~\ref{ass:material}~b)(ii), whence the proof is omitted, see  \cite{You13} for details.

\begin{lemma}\label{lem:A11mon}
Let $\mathit{\Omega}\subset\mathbb{R}^3$ be a~bounded Lipschitz domain and let the operator $\cAl_{11}$ be defined as in \eqref{eq:A11} with $\nu$ satisfying Assumption~\textup{\ref{ass:material}~b)(ii)}. Then for all functions
\mbox{$\bm{A}_1,\bm{A}_2\in X_0(\curl,\mathit{\Omega},\mathit{\Omega}_C)$}, the operator $\cAl_{11}$ fulfills
\begin{equation}
\bigl\langle {\bm{A}}_1-{\bm{A}}_2,\cAl_{11}({\bm{A}}_1)-\cAl_{11}({\bm{A}}_2) \bigr\rangle
\geq m_\nu \, \|\nabla\times ({\bm{A}}_1-{\bm{A}}_2)\|_{L^2(\mathit{\Omega};\mathbb{R}^3)}^2,
\label{eq:A11diss}
\end{equation}
where $m_\nu$ is the monotonicity constant of $\nu$.
\end{lemma}

We now present the uniqueness result.

\begin{theorem}[Uniqueness of the solutions] \label{thm:uniqueness}
Let $\mathit{\Omega}\subset\mathbb{R}^3$ with subdomains $\mathit{\Omega}_C$ and $\mathit{\Omega}_I$ satisfy Assumption~\textup{\ref{ass:omega}}, and let $X(\mathit{\Omega},\mathit{\Omega}_C)$ be defined as in \eqref{eq:statespace}. Further, let the initial and winding functions be as in Assumption~\textup{\ref{ass:init}} and  the material parameters as in Assumption~\textup{\ref{ass:material}}. Let $T>0$ be fixed and let $\bm{v}\in L^2([0,T];\mathbb{R}^m)$ be a given voltage.
Then the coupled MQS system \eqref{eq:MQS} admits at most one weak solution $(\bm{A},\bm{i})$ on $[0,T]$.
\end{theorem}

\begin{proof}
Assume that  $(\bm{A}_k,\bm{i}_k)$,  $k=1$, $2$, are two weak solutions of the coupled MQS system \eqref{eq:MQS}. Consider the operators $\cEl_{ij}$ and $\cAl_{ij}$ as defined in \eqref{eq:EAiiop}.
Note that $\cAl_{12}$ is the adjoint of $\cEl_{21}$, that is,
$\cAl_{12}^{}=\cEl_{21}^*$,
and that $\cEl_{11}^* \cEl_{11}^{}:X(\mathit{\Omega},\mathit{\Omega}_C)\to X(\mathit{\Omega},\mathit{\Omega}_C)$ and $\cAl_{22}^{} \cAl_{22}^*:\mathbb{R}^m\to \mathbb{R}^m$ are both self-adjoint and positive.
Then
\begin{equation}\label{eq:opdae}
\begin{aligned}
\tfrac{\rm d}{{\rm d}t}\cEl_{11}^* \cEl_{11}^{} \bm{A}_k(t) =&\,- \cAl_{11}(\bm{A}_k(t)) + \cEl_{21}^* \bm{i}_k(t),\\
\tfrac{\rm d}{{\rm d}t}\cEl_{21} \bm{A}_k(t) = & - \cAl_{22}^*\cAl_{22}^{} \bm{i}_k(t) + \bm{v}(t),\\
\cEl_{11}^* \cEl_{11}^{} \bm{A}_k(0)=&\,\cEl_{11}^* \cEl_{11}^{} \bm{A}_0,\\
\cEl_{21} \bm{A}_k(0)=&\,\cEl_{21} \bm{A}_0
\end{aligned}
\end{equation}
for $k=1$, $2$. Resolving the second equation in \eqref{eq:opdae} for
\begin{equation}
\bm{i}_k(t) = - \tfrac{\rm d}{{\rm d}t}(\cAl_{22}^*\cAl_{22}^{})^{-1}\cEl_{21}^{}\bm{A}_k(t) +
(\cAl_{22}^*\cAl_{22}^{})^{-1}\bm{v}(t)
\label{eq:iM}
\end{equation}
 and substituting it into the first one, we obtain, for $k=1,2$,
\begin{equation}
\tfrac{\rm d}{{\rm d}t}\bigl(\cEl_{11}^* \cEl_{11}^{}\! + \cEl_{21}^* (\cAl_{22}^*\cAl_{22}^{})^{-1}\cEl_{21}^{}\bigr)\bm{A}_k(t) =- \cAl_{11}^{}(\bm{A}_k(t))\, +\, \cEl_{21}^* (\cAl_{22}^*\cAl_{22}^{})^{-1}\bm{v}(t).
\label{eq:schur}
\end{equation}
Using \eqref{eq:schur} and Lemma~\ref{lem:A11mon},  we obtain for \mbox{$0<t_0\leq t\leq T$} that
\begin{equation}\begin{array}{l}
\displaystyle{\int_{{t_0}}^{{t}}\bigl\langle \bm{A}_1(\tau)-\bm{A}_2(\tau), \tfrac{\rm d}{{\rm d}\tau}(\cEl_{11}^*\cEl_{11}^{} + \cEl_{21}^* (\cAl_{22}^*\cAl_{22}^{})^{-1} \cEl_{21}^{}) (\bm{A}_1(\tau)-\bm{A}_2(\tau))\bigr\rangle\, {\rm d}\tau}\\
\displaystyle{\qquad\quad = - \int_{{t_0}}^{{t}}\bigl\langle \bm{A}_1(\tau)-\bm{A}_2(\tau), \cAl_{11}(\bm{A}_1(\tau))-\cAl_{11}(\bm{A}_2(\tau))\bigr\rangle \,{\rm d}\tau}\\
\displaystyle{\qquad\quad \leq -  m_\nu \int_{{t_0}}^{{t}}\|\nabla\times ({\bm{A}}_1(\tau)-{\bm{A}}_2(\tau))\|_{L^2(\mathit{\Omega};\mathbb{R}^3)}^2 \,{\rm d}\tau.}
\end{array}
\label{eq:est2}
\end{equation}
On the other hand, the self-adjointness of $\cEl_{11}^* \cEl_{11}^{}$ and $\cAl_{22}^*\cAl_{22}^{}$ together with the product rule imply that
\begin{equation}
\begin{array}{l}
\displaystyle{\int_{{t_0}}^{{t}}\bigl\langle \bm{A}_1(\tau)-\bm{A}_2(\tau), \tfrac{\rm d}{{\rm d}\tau}(\cEl_{11}^* \cEl_{11}^{} + \cEl_{21}^* (\cAl_{22}^*\cAl_{22}^{})^{-1}\cEl_{21}^{} ) (\bm{A}_1(\tau)-\bm{A}_2(\tau))\bigr\rangle\, {\rm d}\tau}\\[4mm]
\displaystyle{\quad =\frac12\bigl\langle \bm{A}_1(t)-\bm{A}_2(t), (\cEl_{11}^* \cEl_{11}^{} + \cEl_{21}^* (\cAl_{22}^*\cAl_{22}^{})^{-1}\cEl_{21}^{}) (\bm{A}_1(t)-\bm{A}_2(t))\bigr\rangle}\\[4mm]
\displaystyle{\qquad-\frac12\bigl\langle \bm{A}_1(t_0)-\bm{A}_2(t_0), (\cEl_{11}^* \cEl_{11}^{} + \cEl_{21}^*(\cAl_{22}^*\cAl_{22}^{})^{-1}\cEl_{21}) (\bm{A}_1(t_0)-\bm{A}_2(t_0))\bigr\rangle}\\[4mm]
\displaystyle{\quad =\frac{\sigma_C}2\!\int_{\mathit{\Omega}_C}\|\bm{A}_1(t)-\bm{A}_2(t)\|_2^2\,{\rm d}\xi+\frac12\Bigr\|R^{-1/2}\int_{\mathit{\Omega}}\!\chi^\top(\bm{A}_1(t)-\bm{A}_2(T))\,{\rm d}\xi\Bigl\|_2^2}\\[4mm]
\displaystyle{\qquad -\frac{\sigma_C}2\!\int_{\mathit{\Omega}_C}\|\bm{A}_1(t_0)-\bm{A}_2(t_0)\|_2^2\,{\rm d}\xi-\frac{1}2\Bigl\|R^{-1/2}\int_{\mathit{\Omega}}\!\chi^\top(\bm{A}_1(t_0)-\bm{A}_2(t_0))\,{\rm d}\xi\Bigr\|_2^2.}
\end{array}\label{eq:est3}\end{equation}
Since
$\bm{A}_1$ and $\bm{A}_2$ satisfy the continuity properties in Definition~\ref{def:sol}~\ref{item:sol1}),
\ref{item:sol2})
and the initial conditions
\[\begin{aligned}
& (\sigma\bm{A}_1)(0)=\, \sigma\bm{A}_0=(\sigma\bm{A}_2)(0) , \\
& \int_{\mathit{\Omega}}\chi^\top\bm{A}_1(0)\,{\rm d}\xi=\,\int_{\mathit{\Omega}}\chi^\top\bm{A}_0\,{\rm d}\xi=\int_{\mathit{\Omega}}\chi^\top\bm{A}_2(0)\,{\rm d}\xi, \qquad
\end{aligned}
\]
then taking the limit ${t_0}\to 0+$, we obtain from \eqref{eq:est2} and \eqref{eq:est3} that
\[\begin{aligned}
\frac{\sigma_C}2\int_{\mathit{\Omega}_C}&\|\bm{A}_1(t)-\bm{A}_2(t)\|_2^2\,{\rm d}\xi+\frac{1}2\Bigl\|R^{-1/2}\int_{\mathit{\Omega}}\chi^\top(\bm{A}_1(t)-\bm{A}_2(t))\,{\rm d}\xi\Bigr\|_2^2\\
&\qquad \leq -  m_\nu \int_{0}^{t}\|\nabla\times ({\bm{A}}_1(\tau)-{\bm{A}}_2(\tau))\|_{L^2(\mathit{\Omega};\mathbb{R}^3)}^2 \,{\rm d}\tau \leq 0.
\end{aligned}\]
Since the left-hand side is nonnegative and the right-hand side is nonpositive, we obtain that
$\|{\bm{A}}_1(t)-{\bm{A}}_2(t)\|_{L^2(\mathit{\Omega}_C;\mathbb{R}^3)}=0$ and $\|\nabla\times ({\bm{A}}_1(t)-{\bm{A}}_2(t))\|_{L^2(\mathit{\Omega};\mathbb{R}^3)}=0$ for almost all $t\in [0,T]$.
Then the estimate \eqref{eq:curlest} in Lemma~\ref{lem:denscoerc} implies \mbox{$\bm{A}_1(t)=\bm{A}_2(t)$} for almost all $t\in [0,T]$.
As a consequence, we obtain from \eqref{eq:iM} that \mbox{$\bm{i}_1(t)=\bm{i}_2(t)$} for almost all $t\in{[0,T]}$. This completes the proof.
\end{proof}

\section{The Magnetic Energy}
\label{sec:energy}

Our existence and regularity results for the coupled MQS model \eqref{eq:MQS} rely on the observation that this model is a~special instance of a~differential-algebraic gradient system. In this system, the magnetic energy  plays a~central role. Therefore, in the following, we define the magnetic energy for the coupled MQS system \eqref{eq:MQS} and collect some of its properties.

First, however, we introduce some useful general concepts and notation.

\begin{definition}\label{def:phi}
Let $X$ and $Z$ be Hilbert spaces, and let $\varphi:X\to \mathbb{R}\cup\{\infty\}$ be a~function with values in the extended real line. We call
$D(\varphi)= \varphi^{-1}[0,\infty)$
the {\em effective domain} of $\varphi$, and we say that $\varphi$ is {\em proper}, if its effective domain is nonempty. The function $\varphi$ is {\em convex}, if
\[
\forall\,x_1,x_2\in X, \lambda\in [0,1]:\;\; \varphi(\lambda x_1+(1-\lambda)x_2)\leq \lambda \varphi(x_1)+(1-\lambda)\varphi(x_2) ,
\]
and it is {\em lower semicontinuous}, if for every $\lambda\in\mathbb{R}$, the sublevel set $\varphi^{-1}[0,\lambda]$ is closed in~$X$. We say that $\varphi$ is {\em coercive} if for every $\lambda\in\mathbb{R}$ the sublevel set  $\varphi^{-1}[0,\lambda]$ is bounded.
Finally, given $\mathcal{E}\in \mathcal{L}(X,Z)$, we say that $\varphi$ is {\em $\mathcal{E}$-elliptic} if there exists $\omega\in\mathbb{R}$ such that the shifted functional
\[
\begin{aligned}
\varphi_\omega:&&X\to &\,\mathbb{R} \cup\{\infty\},\\&&x\mapsto&\,\tfrac\omega2\|\mathcal{E} x\|_Z^2+\varphi(x)
\end{aligned}
\]
is convex and coercive.

The~relation
    \[
    \partial\varphi=\!\left\{ (x,q)\!\in\! X\!\times\! X \; : \; \displaystyle{x\!\in\! D(\varphi) \text{ and } \lim_{\lambda\searrow0}\frac{\varphi(x\!+\!\lambda v)\!-\!\varphi(x)}{\lambda}\geq \langle q,v\rangle_X \text{ for all }
    v\!\in\! X}\right\}
    \]
on $X$ is called {\em subgradient of $\varphi$}. For $x\in X$, we write
\[
\partial\varphi(x)=\left\{q\in X\enskip :\enskip (x,q)\in \partial\varphi\right\} ,
\]
and we call
\[
D(\partial\varphi)=\left\{x\in X\enskip : \enskip \exists\; q\in X\text{ such that }(x,q)\in\partial\varphi\right\}
\]
the {\em domain} of the subgradient $\partial\varphi$.
\end{definition}

Now, starting from the magnetic reluctivity
$\nu$ as in Assumption~\ref{ass:material}~b), consider the~function $\vartheta:\mathit{\Omega}\times\mathbb{R}_{\ge0}\to\mathbb{R}_{\ge0}$ defined by
\begin{equation}
\vartheta(\xi,\varrho)=\frac{1}{2}\int_0^\varrho \nu(\xi,\sqrt{\zeta})\, {\rm d}\zeta = \int_0^{\sqrt{\varrho}} \nu(\xi,\zeta)\zeta\,{\rm d}\zeta.
\label{eq:gamma}
\end{equation}
Using this function, we further define the functional $E:X(\mathit{\Omega},\mathit{\Omega}_C)\to \mathbb{R} \cup\{\infty\}$ by
\begin{equation}\label{eq:varphiA}
E(\bm{A}{})=\begin{cases}\displaystyle{\int_{\mathit{\Omega}} \vartheta\bigl(\xi,\|\nabla\times \bm{A}(\xi )\|_2^2\bigr)\,{\rm d}\xi}\; & \text{if } \bm{A}{}\in X_0(\curl,\mathit{\Omega},\mathit{\Omega}_C),\\
\infty\; & \text{else}.\end{cases}
\end{equation}
Note that for the magnetic vector potential $\bm{A}$, the function
$\vartheta\bigl(\xi,\|\nabla\times \bm{A}(\xi,t)\|_2^2\bigr)$ is the {\em magnetic energy density}, and
$E$ describes
the {\em magnetic energy} \cite{HauM89}. This is a~special kind of energy function in the sense that its effective domain \mbox {$D(E)= X_0(\curl,\mathit{\Omega},\mathit{\Omega}_C)$} is a~vector space, actually a~Hilbert space with the natural inner product induced from $H_0(\curl,\mathit{\Omega})$.

In the following, we collect some properties of the magnetic energy, where we further use the notions of {\em G\^{a}teaux differentiability} and {\em G\^{a}teaux derivative} as introduced in \cite[Definition~4.5]{Zeid86}.

\begin{proposition}[Properties of the magnetic energy function]\label{prop:enfun}
Let \mbox{$\mathit{\Omega}\subset\mathbb{R}^3$} with subdomains $\mathit{\Omega}_C$ and $\mathit{\Omega}_I$ satisfy Assumption~\textup{\ref{ass:omega}}, and let $X(\mathit{\Omega},\mathit{\Omega}_C)$ and $X_0(\curl,\mathit{\Omega},\mathit{\Omega}_C)$ be as in \eqref{eq:statespace} and \eqref{eq:statespace2}, respectively. Further, let the material parameters satisfy Assumption~\textup{\ref{ass:material}},
let  $\vartheta$ be as in \eqref{eq:gamma}, and let the magnetic energy $E$ be defined as in~\eqref{eq:varphiA}. Then the following statements hold: \vspace*{-1mm}
\begin{enumerate}[a)]
\item\label{prop:enfun-1} For all $\bm{A}_1,\bm{A}_2\in {D(E)}$,
$$
\arraycolsep=2pt
\begin{array}{ll}
|{E}(&\!\bm{A}_1)-{E}(\bm{A}_2)|\\
& \leq\displaystyle{\frac{L_\nu}2\bigl(\|\nabla\times\bm{A}_1\|_{L^2(\mathit{\Omega};\mathbb{R}^3)}+\|\nabla\times\bm{A}_2\|_{L^2(\mathit{\Omega};\mathbb{R}^3)}\bigr)\, \|\nabla\times(\bm{A}_1-\bm{A}_2)\|_{L^2(\mathit{\Omega};\mathbb{R}^3)},}
\end{array}
$$
where $L_\nu$ is the Lipschitz constant of $\nu$.

\item\label{prop:enfun-2} For all $\bm{A}\in {D(E)}$,
\begin{equation}\label{eq:estE}
\frac{m_\nu}2\|\nabla\times\bm{A}\|^2_{L^2(\mathit{\Omega};\mathbb{R}^3)} \leq {E}(\bm{A})
\leq \frac{L_\nu}2\|\nabla\times\bm{A}\|^2_{L^2(\mathit{\Omega};\mathbb{R}^3)},
\end{equation}
where $m_\nu$ and $L_\nu$ are the monotonicity and Lipschitz constants of $\nu$.

\item\label{prop:enfun-0} If $E$ is considered as a~mapping from the Hilbert space $X_0(\curl,\mathit{\Omega},\mathit{\Omega}_C)$ to $\R$, then $E$ is G\^{a}teaux differentiable, and for all $\bm{A}\in D(E)=X_0(\curl,\mathit{\Omega},\mathit{\Omega}_C)$,
\begin{equation}
\forall\,\bm{F}\in X_0(\curl,\mathit{\Omega},\mathit{\Omega}_C):\;\; \langle\bm{F},{\rm D}E(\bm{A})\rangle
= \int_{\mathit{\Omega}} \nu (\cdot, \| \nabla\times \bm{A}\|_2) \, (\nabla\times \bm{A}) \cdot (\nabla\times \bm{F})\; d\xi,
\end{equation}
where ${\rm D}E(\bm{A})\in X_0(\curl,\mathit{\Omega},\mathit{\Omega}_C)'$ denotes the G\^{a}teaux derivative of $E$ at $\bm{A}$.

\item\label{prop:enfun1} The magnetic energy $E$ is convex and lower semicontinuous. Its effective domain $D(E)=X_0(\curl,\mathit{\Omega},\mathit{\Omega}_C)$ is dense in $X(\mathit{\Omega},\mathit{\Omega}_C)$.

\item\label{prop:enfun3} The subgradient of $E$ is given by
\begin{multline*}
\partial E=\Bigl\{(\bm{A},\bm{C})\in X(\mathit{\Omega},\mathit{\Omega}_C)
\times X(\mathit{\Omega},\mathit{\Omega}_C)\enskip:\enskip \bm{A}\in X_0(\curl,\mathit{\Omega},\mathit{\Omega}_C)
\text{ and} \Bigr.\\
\Bigl.  \int_{\mathit{\Omega}} \nu(\cdot,\|\nabla \times \bm{A}\|_2) \, (\nabla \times \bm{A}) \cdot (\nabla\times \bm{F} )\,{\rm d}\xi = \int_{\mathit{\Omega}} \bm{C} \cdot \bm{F}\,{\rm d}\xi \Bigr. \\
\Bigl. \text{ for all } \bm{F} \in X_0(\curl,\mathit{\Omega},\mathit{\Omega}_C) \Bigr\}.
\end{multline*}
The subgradient $\partial E$ is single-valued and its domain $D(\partial E)$ is dense in $X(\mathit{\Omega},\mathit{\Omega}_C)$.

\item\label{prop:enfun5}
Let $\mathcal{E}\in \mathcal{L}(X(\mathit{\Omega},\mathit{\Omega}_C),X(\mathit{\Omega},\mathit{\Omega}_C)\times \R^m)$ be defined as
\begin{equation}\label{eq:operatorE}
\mathcal{E}\bm{A}=\Bigl(\sqrt{\sigma}\bm{A},R^{-1/2}\int_\mathit{\Omega} \chi^\top\bm{A}\, {\rm d}\xi\Bigr),
\end{equation}
where $R^{-1/2}$ denotes the inverse of the principal square root of $R$.
Then $ E$ is \mbox{$\mathcal{E}$-elliptic}.
\end{enumerate}
\end{proposition}

\begin{proof}
\ref{prop:enfun-1}) Let  $\bm{A}_1$, $\bm{A}_2\in {D(E)}$. Then using
\eqref{eq:Mnugreater} and the Cauchy-Schwarz inequa\-li\-ty, we obtain
{\allowdisplaybreaks\begin{align*}
|E(\bm{A}_1)\,-&\,E(\bm{A}_2)|\\
\leq \; &\int_\mathit{\Omega}\left|\int_0^{\|\nabla\times\bm{A}_1(\xi)\|_2} \nu(\xi,\zeta)\zeta\,{\rm d}\zeta-\int_0^{\|\nabla\times\bm{A}_2(\xi)\|_2} \nu(\xi,\zeta)\zeta\,{\rm d}\zeta  \right|\,{\rm d}\xi\\
= \;&\int_\mathit{\Omega}\left|\int_{\|\nabla\times\bm{A}_2(\xi)\|_2}^{\|\nabla\times\bm{A}_1(\xi)\|_2} \nu(\xi,\zeta)\zeta\,{\rm d}\zeta\right|\,{\rm d}\xi\\
\leq\;&
L_\nu\int_\mathit{\Omega}\left|\int_{\|\nabla\times\bm{A}_2(\xi)\|_2}^{\|\nabla\times\bm{A}_1(\xi)\|_2}\zeta\,{\rm d}\zeta\right|\,{\rm d}\xi\\
=\;&\frac{L_\nu}2\int_\mathit{\Omega}\Bigl|\|\nabla\times\bm{A}_1(\xi)\|_2^2-\|\nabla\times\bm{A}_2(\xi)\|_2^2\Bigr|\,{\rm d}\xi\\
=\;&\frac{L_\nu}2\!\!\int_\mathit{\Omega}\bigl(\|\nabla\!\times\!\bm{A}_1(\xi)\|_2\!+\!\|\nabla\!\times\!\bm{A}_2(\xi)\|_2\bigr)\bigl|\,\|\nabla\!\times\!\bm{A}_1(\xi)\|_2\!-\!\|\nabla\!\times\!\bm{A}_2(\xi)\|_2\,\bigr|\,{\rm d}\xi\\
\leq\;&\frac{L_\nu}2\int_\mathit{\Omega}\bigl(\|\nabla\times\bm{A}_1(\xi)\|_2+\|\nabla\times\bm{A}_2(\xi)\|_2\bigr)\, \|\nabla\times(\bm{A}_1(\xi)-\bm{A}_2(\xi))\|_2\,{\rm d}\xi\\
\leq\;&\frac{L_\nu}2\bigl(\|\nabla\times\bm{A}_1\|_{L^2(\mathit{\Omega};\mathbb{R}^3)}+\|\nabla\times\bm{A}_2\|_{L^2(\mathit{\Omega};\mathbb{R}^3)}\bigr)\, \|\nabla\times(\bm{A}_1-\bm{A}_2)\|_{L^2(\mathit{\Omega};\mathbb{R}^3)}.
\end{align*}}

\ref{prop:enfun-2})
{This statement follows immediately from \eqref{eq:Mnugreater}.}

\ref{prop:enfun-0}) The proof of this assertion is perhaps tedious but elementary; it mainly relies on the differentiability of $\vartheta$ and growth estimates of $\nu$ (compare with \eqref{eq:Mnugreater}). We omit the details.

\ref{prop:enfun1}) By assumption, for almost every $\xi\in\mathit{\Omega}$, the function $\zeta\mapsto\nu (\xi , \zeta ) \zeta$ is positive and increasing on $\mathbb{R}_{\geq 0}$. As a consequence, by definition of $\vartheta$, for almost every $\xi\in\mathit{\Omega}$, the function $\varrho \mapsto \vartheta (\xi , \varrho^2)$ is increasing and convex on $\mathbb{R}_{\geq 0}$. Using these properties and the triangle inequality for the norm, we can easily show the convexity of $E$. Indeed, for all $\bm{A}_1$, $\bm{A}_2\in D(E)$ and all $\lambda\in [0,1]$, we have
\begin{align*}
E(\lambda\bm{A}_1+(1-\lambda) \bm{A}_2) & = \int_{\mathit{\Omega}} \vartheta (\xi , \| \nabla\times (\lambda\bm{A}_1(\xi)+(1-\lambda) \bm{A}_2(\xi))\|_2^2 ) \; {\rm d}\xi \\
& \leq \int_{\mathit{\Omega}} \vartheta (\xi , ( \lambda \| \nabla\times\bm{A}_1(\xi) \|_2 + (1-\lambda) \| \nabla\times\bm{A}_2(\xi)\|_2 )^2 )\; {\rm d}\xi \\
& \leq \int_{\mathit{\Omega}} ( \lambda \vartheta (\xi , \| \nabla\times\bm{A}_1(\xi) \|_2^2 ) + (1-\lambda) \vartheta (\xi , \| \nabla\times\bm{A}_2(\xi)\|_2^2 ) ) \; {\rm d}\xi \\
& = \lambda E (\bm{A}_1 ) + (1-\lambda ) E(\bm{A}_2 ) .
\end{align*}
In order to prove lower semicontinuity of $E$, let $\lambda\in\mathbb{R}$. From assertion \ref{prop:enfun-1}), or alternatively from \ref{prop:enfun-0}), we see that the magnetic energy is continuous as a mapping from the Hilbert space $X_0(\curl,\mathit{\Omega},\mathit{\Omega}_C)$ to $\R$, and, therefore, the sublevel set
$$
\{ \bm{A}\in X(\mathit{\Omega},\mathit{\Omega}_C)\; |\; E (\bm{A}) + \| \bm{A}\|_{L^2(\mathit{\Omega};\mathbb{R}^3)}^2 \leq \lambda \}
$$ is closed in that space. By convexity of~$E$, this set is also convex, and hence, by Mazur's theorem, this set is weakly closed. Further, by \eqref{eq:estE}, the sublevel set is bounded in $X_0(\curl,\mathit{\Omega},\mathit{\Omega}_C)$, and hence, since every Hilbert space is reflexive, it is weakly compact. By continuity of the embedding of $X_0(\curl,\mathit{\Omega},\mathit{\Omega}_C)$ into $X(\mathit{\Omega},\mathit{\Omega}_C)$, it follows that this sublevel set is also weakly compact in the latter space, and then it is necessarily norm closed there. We have thus proved that the mapping $\bm{A}\mapsto E(\bm{A}) + \| \bm{A}\|_{L^2(\mathit{\Omega};\mathbb{R}^3)}^2$ is lower semicontinuous on $X(\mathit{\Omega},\mathit{\Omega}_C)$. Since $\|\cdot \|_{L^2(\mathit{\Omega};\mathbb{R}^3)}^2$ is continuous on that space, it follows that the magnetic energy itself is lower semicontinuous. The fact that $E$ is densely defined follows from Lemma~\ref{lem:denscoerc}.

Assertion \ref{prop:enfun3}) is a direct consequence of the definition of the subgradient, the fact that the effective domain of $E$ is a linear space, and from assertion \ref{prop:enfun-0}). By \cite[Proposition~1.6]{Bar10}, $D(\partial E)$ and $D( E)$ have the same closure, and hence the subgradient is densely defined, too.

\ref{prop:enfun5}) Let $Z=X(\mathit{\Omega},\mathit{\Omega}_C)\times \R^m$ and let $\mathcal{E}\in\mathcal{L}(X(\mathit{\Omega},\mathit{\Omega}_C),Z)$ be as in \eqref{eq:operatorE}. First, note that $\mathcal{E}$ is well-defined, since by Assumption~\ref{ass:material}~c) the matrix $R$ is symmetric and positive definite. In order to prove that $ E$ is $\mathcal{E}$-elliptic, we show that the shifted functional
$$
\arraycolsep=2pt
\begin{array}{rcl}
 E_{\omega}:X(\mathit{\Omega},\mathit{\Omega}_C)&\to& \mathbb{R}_{\ge0}\cup\{\infty\},\\
 {\bm{A}}& \mapsto & \displaystyle{\tfrac\omega2\|\mathcal{E}{\bm{A}}\|_Z^2+ E({\bm{A}})}
\end{array}
$$
is convex and coercive for all $\omega>0$. Convexity follows from the fact that $E_\omega$ is the sum of two convex functions. To show coercivity, we
use Lemma~\ref{lem:denscoerc}, which states that there exists some $L_C>0$ such that \eqref{eq:curlest} holds for all $\bm{A}\in D( E)$.  Further, by using \eqref{eq:estE}, we obtain
\[\begin{aligned}
{ E_{\omega}}(\bm{A})=&\,{ E}(\bm{A})+\frac{\omega}{2} \|\mathcal{E}\bm{A}\|_{Z}^2\\
=&\;{ E}(\bm{A})+\frac{\omega\sigma_C}{2} \int_{\mathit{\Omega}_C} \|\bm{A}\|_2^2\,{\rm d}\xi +
\frac{\omega}{2} \Bigl\|R^{-1/2}\int_{\mathit{\Omega}} \chi^\top\bm{A}\,{\rm d}\xi\Bigr\|_2^2\\
\geq&\; \frac{m_\nu}2\,\|\nabla\times\bm{A}\|^2_{L^2(\mathit{\Omega};\mathbb{R}^3)}+
\frac{\omega\sigma_C}{2} \|\bm{A}\|_{L^2(\mathit{\Omega}_C;\mathbb{R}^3)}^2\\
\geq&\;c\,\|\bm{A}\|^2_{L^2(\mathit{\Omega};\mathbb{R}^3)}
\end{aligned}\]
with ${c}=\min\{m_\nu,\omega\sigma_C\}/(2L_C)$. This implies the coercivity of $E_\omega$.
\end{proof}

\section{On a~Class of Abstract Differential-Algebraic Gradient Systems}
\label{sec:gradsys}

In this section, we study the solvability of an~abstract differential-algebraic gradient system
 \begin{equation} \label{eq:symmdae}
\mathcal{E}^*f(t)-\mathcal{E}^*\tfrac{\rm d}{{\rm d}t}\mathcal{E} x(t)\in \partial\varphi(x(t)),
\quad\mathcal{E} x(0+)=z_0,
\end{equation}
where $\mathcal{E}\in\mathcal{L}(X,Z)$, $X$ and $Z$ are Hilbert spaces, $f\in L^2([0,T],Z)$,  $\varphi$ is a~densely defined, convex, lower semicontinuous and $\mathcal{E}$-elliptic functional with a~subgradient $\partial\varphi$, and $z_0\in Z$.
We present an~extension of some results from  \cite{Bar10,Bre73} to this more general class of gradient systems, which will be useful in establishing the existence of solutions of the MQS system \eqref{eq:MQS}.

We start by proving an auxiliary lemma providing the concept of $\mathcal{E}$-subgradients.

\begin{lemma}\label{lem:modfun}
Let $X$ and $Z$ be Hilbert spaces, and let $\mathcal{E}\in\mathcal{L}(X,Z)$ have a~dense range. Assume that the functional $\varphi:X\to\mathbb{R} \cup\{\infty\}$ is densely defined, convex, lower semicontinuous and $\mathcal{E}$-elliptic with a~subgradient $\partial\varphi$. Define the functional
 \begin{equation}
\arraycolsep=2pt
\begin{array}{rcl}
\varphi_{\mathcal{E}}: Z & \to &\R, \\
z & \mapsto & \displaystyle{\inf_{x\in \mathcal{E}^{-1}\{z\}}\varphi(x)} .
\end{array}
\label{eq:modfun}
\end{equation}
Then $D(\varphi_{\mathcal{E}})={\mathcal{E}}D(\varphi)$, and $\varphi_{\mathcal{E}}$ is densely defined, convex and lower semicontinuous. Its subgradient is given by
\[
\begin{array}{rl}
\partial\varphi_{\mathcal{E}}=\Bigl\{ (z,g)\in Z\times Z\enskip : \enskip \Bigr.&
\exists\, x\in D(\varphi)\text{ such that }\mathcal{E} x=z \text{ and }\\
& \Bigl.\displaystyle{\lim_{\lambda\searrow0}\frac{\varphi(x+\lambda v)-\varphi(x)}{\lambda}\geq \langle g,\mathcal{E} v\rangle_Z} \text{ for all } v\in X\Bigr\}.
\end{array}
\]
In particular, $(z,g)\in\partial\varphi_{\mathcal{E}}$ if and only if there exists $(x,q)\in\partial\varphi$
such that $\mathcal{E} x=z$, $\mathcal{E}^*g=q$, and $\varphi(x)=\varphi_{\mathcal{E}}(z)$.
\end{lemma}

\begin{proof}
The statement is, except for the assertion on the density of the domain of $\varphi_{\mathcal{E}}$, proven in {\cite[Theorem~2.9]{CHK16}}. The density of the domain of $\varphi_{\mathcal{E}}$ can be inferred from the identity $D(\varphi_{\mathcal{E}})=\mathcal{E} D(\varphi)$, the assumption that $\varphi$ is densely defined and by employing the property that $\mathcal{E}$ has a~dense range.
\end{proof}

Next, we prove the existence and regularity properties of solutions of the abstract differential-algebraic gradient system \eqref{eq:symmdae}. Note that the initial value in \eqref{eq:symmdae} is only in the closure of the range of $\mathcal{E}$, and that the initial condition is to be understood to hold for the continuous representative of $\mathcal{E} x$.

\begin{theorem}\label{lem:symmdae}
Let $X$ and $Z$ be Hilbert spaces, let $\mathcal{E}\in\mathcal{L}(X,Z)$, and let $\widetilde{Z}\subseteq Z$ be the closure of the range of $\mathcal{E}$. Furthermore, let $\varphi:X\to\mathbb{R} \cup\{\infty\}$ be a~densely defined, convex, lower semicontinuous and $\mathcal{E}$-elliptic functional with subgradient $\partial\varphi$. Then for every $T>0$, $f\in L^2({[0,T]};Z)$ and $z_0\in \widetilde{Z}$,
the {abstract differential-algebraic gradient system} \eqref{eq:symmdae}
admits a~solution $x:{[0,T]}\to X$ in the following sense: \vspace*{-1mm}
\begin{romanlist}[a)]
\item\label{item:sola} $\mathcal{E} x \in C([0,T] ;Z) \cap  H^{1}_{\rm loc}((0,T];Z)$ and
$\mathcal{E} x(0+)= z_0$,\vspace*{1mm}
\item\label{item:solc} $x(t)\in D(\partial\varphi)$ and the differential inclusion in \eqref{eq:symmdae} hold for almost all $t\in [0,T]$.
\end{romanlist}
This solution has the following properties:
\begin{align}
\bigl(t\mapsto \varphi(x(t))\bigr)\in&\, L^1([0,T])\cap {W^{1,1}_{\rm loc}((0,T])},
\label{eq:DAEBarb1}\\
\bigl(t\mapsto t\,\varphi(x(t))\bigr)\in&\, L^\infty([0,T]),\label{eq:DAEBarb1.5}\\
\bigl(t\mapsto t^{1/2}\tfrac{\rm d}{{\rm d}t}\mathcal{E} {x}(t)\bigr)\in&\, L^2([0,T];Z),\label{eq:DAEBarb2}
\intertext{and for almost all $0<t_0\leq t_1\leq T$}
\quad  \varphi(x(t_1))-\varphi(x(t_0))&\,=\int_{t_0}^{t_1}\langle f(\tau),\tfrac{\rm d\,}{{\rm d}\tau}\mathcal{E} x(\tau)\rangle_Z\,{\rm d}\tau-\int_{t_0}^{t_1}\|\tfrac{\rm d\,}{{\rm d}\tau} \mathcal{E} {x}(\tau)\|^2_Z\,{\rm d}\tau.\label{eq:DAEBarb5}
\end{align}
If, further, $y_0\in \mathcal{E} D(\varphi)$, then
\begin{align}
\bigl(t\mapsto \varphi(x(t))\bigr)\in&\, W^{1,1}([0,T]), \label{eq:DAEBarb3}\\
\mathcal{E} x\in&\, H^1({[0,T]};Z),\label{eq:DAEBarb4}
\end{align}
and the identity \eqref{eq:DAEBarb5} holds for all $0\leq t_0\leq t_1\leq T$.
\end{theorem}

Before proving this result, we note that, except for the relation \eqref{eq:DAEBarb1.5}, a~proof of Theorem~\ref{lem:symmdae} for the special case $X=Z$ and $\mathcal{E}=I$ can be found in \cite[Th\'eor\`eme~3.6, p.72]{Bre73} or \cite[Theorem~4.11 \& Lemma~4.4]{Bar10}. These results for $\mathcal{E}=I$ are indeed the basis for our proof of Theorem~\ref{lem:symmdae}.

\begin{proof}
Since $\widetilde{Z}\subseteq Z$ is the closure of the range of $\mathcal{E}$, then $\mathcal{E} :X\to\widetilde{Z}$ has dense range, and Lemma~\ref{lem:modfun} implies that the functional $\varphi_{\mathcal{E}}:\widetilde{Z}\to\mathbb{R} \cup\{\infty\}$ in \eqref{eq:modfun} is densely defined, convex and lower semicontinuous. Further,
let $\varPi\in\mathcal{L}(Z)$ be the orthogonal projection onto $\widetilde{Z}$.
Then $\varPi f\in L^2({[0,T]};\widetilde{Z})$, and,  by \cite[Th\'eor\`eme~3.6, p.72]{Bre73} or \cite[Theorem~4.11]{Bar10}, the {gradient system}
\begin{equation}
\varPi{f}(t)-\tfrac{\rm d}{{\rm d}t}z(t)\in \partial\varphi_{\mathcal{E}}(z(t)), \quad z(0) = z_0\label{eq:resolODE}
\end{equation}
admits a~unique solution $z\in C ([0,T] ; \widetilde{Z}) \cap H^{1}_{\rm loc}((0,T];{\widetilde{Z}})$ in the sense that $z(0) = z_0$, $z(t)\in D(\varphi_{\mathcal{E}} )$ and the differential inclusion in \eqref{eq:resolODE} hold for almost all $t\in [0,T]$.
Moreover, this solution has the following properties:
\begin{align}
\bigl(t\mapsto \varphi_{\mathcal{E}}(z(t))\bigr)\in&\, L^1([0,T]),\label{eq:Barb1}\\
\bigl(t\mapsto t^{1/2}\tfrac{\rm d}{{\rm d}t}z(t)\bigr)\in&\, L^2([0,T];{\widetilde{Z}}).\label{eq:Barb2}
\end{align}
The fact that the differential inclusion \eqref{eq:resolODE} holds for almost all $t\in [0,T]$ is equivalent to saying that for almost all $t\in{[0,T]}$,
\[
\bigl(z(t),\varPi f(t)-\tfrac{\rm d}{{\rm d}t}z(t)\bigr)\in \partial\varphi_{\mathcal{E}}.
\]
Then, by Lemma~\ref{lem:modfun}, there exist functions $x$, $w:[0,T]\to X$ such that,
for almost all $t\in[0,T]$, $(x(t),w(t))\in \partial\varphi$, $\mathcal{E} x(t)=z(t)$ and $w(t)=\mathcal{E}^*\varPi f(t)-\mathcal{E}^*\tfrac{\rm d}{{\rm d}t}z(t)$. In particular, $\mathcal{E} x=z\in C([0,T] ; \widetilde{Z} ) \cap H^1_{\rm loc}((0,T];\widetilde{Z})$, that is, $z$ is a continuous representative of $\mathcal{E} x$, and therefore $\mathcal{E} x(0+) = \lim_{t\to 0+} \mathcal{E} x(t) = \lim_{t\to 0+} z(t) = z_0$. In addition, by using that $\varPi\mathcal{E} = \mathcal{E}$ together with $\varPi=\varPi^*$ implies $\mathcal{E}^*=\mathcal{E}^*\varPi$, we have
\begin{align*}
  \mathcal{E}^*f(t)-\mathcal{E}^*\tfrac{\rm d}{{\rm d}t} \mathcal{E} x(t)& =
\mathcal{E}^*\varPi f(t)-\mathcal{E}^*\tfrac{\rm d}{{\rm d}t} \mathcal{E} x(t)\\
& =\mathcal{E}^*\varPi f(t)-\mathcal{E}^*\tfrac{\rm d}{{\rm d}t}z(t)=w(t)\in \partial\varphi(x(t)),
\end{align*}
that is, \ref{item:sola}) and \ref{item:solc}) hold.
Moreover, \eqref{eq:Barb2} implies \eqref{eq:DAEBarb2}. Using Lemma~\ref{lem:modfun}, we obtain
\begin{equation}
\varphi_{\mathcal{E}}(z(t))=\varphi(x(t)) \quad\text{for almost all } t\in[0,T] ,
\label{eq:phiErel}
\end{equation}
and then \eqref{eq:Barb1} leads to $\bigl(t\mapsto \varphi(x(t))\bigr)\in\, L^1([0,T])$.
An application of \cite[Lemme~3.3, p.73]{Bre73} or \cite[Lemma~4.4]{Bar10} gives $\bigl(t\mapsto \varphi_{\mathcal{E}}(z(t))\bigr)\in\, W^{1,1}_{\rm loc}((0,T])$, whence \eqref{eq:phiErel} implies $\left(t\mapsto \varphi(x(t))\right)\in\, W^{1,1}_{\rm loc}((0,T])$, which shows \eqref{eq:DAEBarb1}.

In order to prove \eqref{eq:DAEBarb5}, recall that the property $z\in H^1_{\rm loc}((0,T];\tilde{Z})$ together with \cite[Lemma~4.4]{Bar10} implies that the weak derivative of $\bigl(t\mapsto \varphi_{\mathcal{E}}(z(t))\bigr)$ fulfills
\[
\tfrac{\rm d}{{\rm d}t}\varphi_{\mathcal{E}}(z(t))=\bigl\langle \varPi f(t),\tfrac{\rm d}{{\rm d}t}z(t) \bigr\rangle_Z-\bigl\|\tfrac{\rm d}{{\rm d}t}z(t)\bigr\|_Z^2\enskip\text{ for almost all }t\in[0,T].
\]
Then, for $0<t_0\leq t_1\leq T$, an integration over $[t_0,t_1]$ gives
\[
\varphi_{{\mathcal{E}}}(z(t_1))-\varphi_{{\mathcal{E}}}(z(t_0))=\int_{t_0}^{t_1}\bigl\langle \varPi f(\tau),\tfrac{\rm d}{{\rm d}\tau}z(\tau) \bigr\rangle_Z-\bigl\|\tfrac{\rm d}{{\rm d}\tau}z(\tau)\bigr\|^2_Z\,{\rm d}\tau.
\]
Using the equality $\mathcal{E} x=z$, \eqref{eq:phiErel} and the self-adjointness of $\varPi$, we obtain \eqref{eq:DAEBarb5}.

For the proof of \eqref{eq:DAEBarb1.5}, let $t\in(0,T]$. Then \eqref{eq:DAEBarb5} together with
Young's inequality
\cite[p.~53]{Alt16}
leads to
\[\begin{aligned}
t\,\varphi(x(t))
=&\,t\,\varphi(x(T))+t\,
\int_{t}^{T}\|\tfrac{\rm d\,}{{\rm d}\tau} \mathcal{E}{x}(\tau)\|^2_Z\,{\rm d}\tau-
t\,\int_{t}^{T}\langle f(\tau),\tfrac{\rm d\,}{{\rm d}\tau}\mathcal{E} x(\tau)\rangle_Z\,{\rm d}\tau\\
\leq&\,t\,\varphi(x(T))+t\,
\int_{t}^{T}\|\tfrac{\rm d\,}{{\rm d}\tau} \mathcal{E}{x}(\tau)\|^2_Z\,{\rm d}\tau+
\frac{t}2\,\int_{t}^{T} \Bigl( \|\tfrac{\rm d\,}{{\rm d}\tau}\mathcal{E} x(\tau)\|_{Z}^2+\|f(\tau)\|_{Z}^2\Bigr) \,{\rm d}\tau\\
\leq&\,T\,\varphi(x(T))+
\frac32\int_{t}^{T}\|\tau^{1/2}\tfrac{\rm d\,}{{\rm d}\tau} \mathcal{E}{x}(\tau)\|^2_Z\,{\rm d}\tau+
\frac{T}2\,\int_{t}^{T}\|f(\tau)\|_{{Z}}^2\,{\rm d}\tau\\
\leq&\,T\,\varphi(x(T))+
\frac32\int_{0}^{T}\|\tau^{1/2}\tfrac{\rm d\,}{{\rm d}\tau} \mathcal{E}{x}(\tau)\|^2_Z\,{\rm d}\tau+
\frac{T}2\,\int_{0}^{T}\|f(\tau)\|_{Z}^2\,{\rm d}\tau.
\end{aligned}\]
As the latter expression is independent of $t$ and finite by $f\in L^2({[0,T]};Z)$ and the
already proved relation \eqref{eq:DAEBarb2}, we obtain \eqref{eq:DAEBarb1.5}.

It remains to prove the statements under the additional assumption that the initial value fulfills $z_0\!\in\! \mathcal{E} D(\varphi)$ or, by Lemma~\ref{lem:modfun}, $z_0\in D(\varphi_{\mathcal{E}})$. Then we can apply \cite[Theorem~4.11]{Bar10} to obtain that the solution $z$ of \eqref{eq:resolODE} fulfills
\begin{align}
\bigl(t\mapsto \varphi_{\mathcal{E}}(z(t))\bigr)\in&\, W^{1,1}([0,T]),\label{eq:Barbsmooth1}\\
z\in&\, H^{1}({[0,T]};{\widetilde{Z})}.\label{eq:Barbsmooth2}
\end{align}
Then $\mathcal{E} x=z$ together with \eqref{eq:phiErel}, \eqref{eq:Barbsmooth1} and \eqref{eq:Barbsmooth2} implies \eqref{eq:DAEBarb3} and \eqref{eq:DAEBarb4}. The statement that the identity \eqref{eq:DAEBarb5} further holds for all $0\leq t_0\leq t_1\leq T$ can be concluded from \eqref{eq:Barbsmooth1} and the argumentation of the proof of \eqref{eq:DAEBarb5} in the case $t_0>0$.
\end{proof}

\begin{remark}
In numerical analysis of finite-dimensional differential-algebraic equations, the notion of {\em (differentiation) index} plays a~fundamental role \textup{\cite{LamMT13,KunM06}}. That is, the number of differentiations needed until an ordinary differential equation is obtained. Though there exist several attempts to generalize the index to infinite-dimensional differential-algebraic equations \textup{\cite{RT05,Rei06,Rei07,TrosWaur18,TrosWaur19}}, these approaches have in common that they are applicable to a~rather limited class, which in general excludes equations of type \eqref{eq:symmdae} even when $\partial\varphi$ is linear and single-valued.
On the other hand, system \eqref{eq:symmdae} has intrinsic properties which are - in the finite-dimensional case - only fulfilled by differential-algebraic equations with index at most one. Namely, it follows from \textup{\cite[Theorem 3.53]{LamMT13}} that a~finite-dimensional differential-algebraic equation of type $\tfrac{\rm d}{{\rm d}t}\cEl{x}=a(t,x)$ with $\cEl\in\R^{n\times n}$ has index at most one if and only if for any $x_0\in\R^n$, there exists a~solution which fulfills the initial condition $\cEl x(0)=\cEl x_0$. Note that, by~Theorem~\textup{\ref{lem:symmdae}}, system \eqref{eq:symmdae} has this property of unrestricted initializability.
\end{remark}

\newpage
Let us now apply Theorem~\ref{lem:symmdae} to the following special differential-algebraic gradient system
\begin{equation}\label{eq:opdaegrad}
\begin{aligned}
\cEl_{21}^*x_2(t)-\cEl_{11}^*\tfrac{d}{dt}\cEl_{11} x_1(t)\in & \,\partial\varphi(x_1(t)),\\
\cAl_{22}\cAl_{22}^* x_2(t)+\tfrac{d}{dt}\cEl_{21} x_1(t)  = & \, u(t),\\
\cEl_{11} x_1(0+)= & \,z_{1,0},\\
\cEl_{21} x_1(0+)= & \,z_{2,0}
\end{aligned}
\end{equation}
with a~given function $u:[0,T]\to U$ and given initial values $z_{1,0}\in Y$, $z_{2,0}\in U$. In the subsequent section, we show that, by involving the magnetic energy and its subgradient, the coupled MQS system \eqref{eq:MQS} fits into this abstract framework, and we apply our results for \eqref{eq:opdaegrad} to prove the existence of solutions together with some further regularity results.

The following corollary establishes the existence result for system \eqref{eq:opdaegrad}.

\begin{corollary}\label{thm:infDAEsys}
Let $X$, $Y$ and $U$ be Hilbert spaces and let $\cAl_{22}\in\mathcal{L}({U})$  have a~bounded inverse,
$\cEl_{11}\in\mathcal{L}(X,Y)$, \mbox{$\cEl_{21}\in\mathcal{L}(X,U)$},
 and $\varphi:X\to\mathbb{R} \cup\{\infty\}$
be densely defined, convex, lower semicontinuous and $\mathcal{E}$-elliptic for
$\mathcal{E}\in\mathcal{L}(X,Y\times U)$ given by
    \begin{equation}
\mathcal{E} x=(\cEl_{11}x, \cAl_{22}^{-1}\cEl_{21}x).\label{eq:Edef}
\end{equation}
Further, let
 $T>0$, $u\in L^2({[0,T]};U)$ and $(z_{1,0},z_{2,0})$ belong to the closure of the range of $\mathcal{E}$ in $Y\times U$. Then the {abstract differential-algebraic gradient system} \eqref{eq:opdaegrad}
has a~solution $(x_1,x_2):[0,T]\to X\times U$ in the following sense: \vspace*{-1mm}
\begin{romanlist}[a)]
\item\label{item:solsysa1} $\cEl_{11}x_1 \in C([0,T]; Y) \cap H^{1}_{\rm loc}((0,T];Y)$ and
$\cEl_{11}x_1 (0+) = z_{1,0}$;\vspace*{1mm}
\item\label{item:solsysa2} $\cEl_{21}x_1 \in C([0,T]; U) \cap H^{1}_{\rm loc}((0,T];U)$ and
$\cEl_{21}x_1 (0+) = z_{2,0}$;\vspace*{1mm}
\item\label{item:solsysb3} $x_2\in L^{2}_{\rm loc}((0,T],U)$;\vspace*{1mm}
\item\label{item:solsysc} $x_1(t)\in D(\partial\varphi)$, and the differential inclusion as well as the differential equation in \eqref{eq:opdaegrad} are satisfied for almost all $t\in [0,T]$.
\end{romanlist}
This solution has the following properties:
\begin{align}
\bigl(t\mapsto \varphi(x_1(t))\bigr)\in&\, L^1({[0,T]})\cap W^{1,1}_{\rm loc}((0,T]),\label{eq:DAEBarbsys1}\\
\bigl(t\mapsto t\,\varphi(x_1(t))\bigr)\in &\, L^\infty([0,T]),\label{eq:DAEBarbsys1.5}\\
\bigl(t\mapsto t^{1/2}\tfrac{\rm d}{{\rm d}t}\cEl_{11}{x}_1(t)\bigr)\in&\, L^2([0,T];Y),\label{eq:DAEBarbsys2}\\
\bigl(t\mapsto t^{1/2}\tfrac{\rm d}{{\rm d}t}\cEl_{21}{x}_1(t)\bigr)\in&\, L^2([0,T];U),\label{eq:DAEBarbsys3}\\
\bigl(t\mapsto t^{1/2}{x}_2(t)\bigr)\in&\, L^2([0,T];U),\label{eq:DAEBarbsys3.5}
\end{align}
and for all $0<t_0\leq t_1\leq T$,
\begin{align}
\quad  \varphi(x_1(t_1))-\varphi(x_1(t_0))&\,= \int_{t_0}^{t_1}\langle x_2(\tau),u(\tau)\rangle_U \,{\rm d}\tau
-\int_{t_0}^{t_1}\|\cAl_{22}^*{x}_2(\tau)\|^2_{{U}}\,{\rm d}\tau\label{eq:DAEBarbsys4}\\
&\qquad-\int_{t_0}^{t_1}\|\tfrac{\rm d}{{\rm d}\tau} \cEl_{11}{x}_1(\tau)\|^2_Y\,{\rm d}\tau.\nonumber
\end{align}
If, further, $(z_{1,0},\cAl_{22}^{-1}z_{2,0})\in \mathcal{E} D(\varphi)$, then
\begin{align}
\bigl(t\mapsto \varphi(x_1(t))\bigr)\in&\, W^{1,1}([0,T]),\label{eq:DAEBarbsys5}\\
\cEl_{11}{x}_1\in&\, H^1([0,T];Y),\label{eq:DAEBarbsys6}\\
\cEl_{21}{x}_1\in&\, H^1([0,T];U),\label{eq:DAEBarbsys7}\\
{x}_2\in&\, L^2([0,T];U),\label{eq:DAEBarbsys8}
\end{align}
and the identity {\eqref{eq:DAEBarbsys4}} holds for all $0\leq t_0\leq t_1\leq T$.
\end{corollary}

\begin{proof}
Let $\mathcal{E}\in\mathcal{L}(X,Y\times U)$ be as in \eqref{eq:Edef}.
By the assumption, the functional $\varphi$ is $\mathcal{E}$-elliptic. Consider the function
\begin{equation}
f=\bigl(0,\cAl_{22}^{-1}u\bigr)\in L^2([0,T];Y\times U).\label{eq:fdaedef}
\end{equation}
Theorem~\ref{lem:symmdae} implies that the abstract differential-algebraic gradient system \eqref{eq:symmdae}
with $\mathcal{E}$ as in \eqref{eq:Edef} and $z_0=(z_{1,0},\mathcal{A}_{22}^{-1}z_{2,0})$
has a~{solution} $x:{[0,T]}\to X$ in the sense that
$\mathcal{E} x \in C([0,T] ; Y\times U) \cap H^{1}_{\rm loc}((0,T];Y\times U)$,
$\mathcal{E} x(0+) = (z_{1,0}, \mathcal{A}_{22}^{-1}z_{2,0})$, and $x(t)\in D(\partial\varphi)$
and the differential inclusion in \eqref{eq:symmdae} hold for almost all $t\in [0,T]$.

We now consider $(x_1,x_2):[0,T]\to X\times U$ with
\begin{equation}
\begin{aligned}
x_1(t)=&\,x(t),\\
x_2(t)=&\,(\cAl_{22}^{-1})^*\cAl_{22}^{-1}\bigl(u(t)-\tfrac{\rm d}{{\rm d}t} \cEl_{21}x(t)\bigr).
\end{aligned}\label{eq:subsdaeref}
\end{equation}
Then the above properties of the solution $x$ imply that
$\cEl_{11}x_1:[0,T]\to Y$ and \mbox{$\cEl_{21}x_1:[0,T]\to U$}
are both continuous with $\cEl_{11}x_1(0+)=z_{1,0}$, $\cEl_{21}x_1(0+)=z_{2,0}$,
$\cEl_{11}x_1\in H^{1}_{\rm loc}((0,T],Y)$, $\cEl_{21}x_1\in H^{1}_{\rm loc}((0,T],U)$
and $x_2\in L^{2}_{\rm loc}((0,T],U)$. Furthermore, for almost all \mbox{$t\!\in\! [0,T]$}, $x_1(t)\in D(\partial\varphi)$ and
\[\begin{aligned}
\cEl_{21}^*x_2(t)\!-\!\cEl_{11}^*\tfrac{\rm d}{{\rm d}t}\cEl_{11} x_1(t)
\overset{\eqref{eq:subsdaeref}}{=}&\cEl_{21}^*(\cAl_{22}^{-1})^*\cAl_{22}^{-1}\bigl(u(t)\!-\!\tfrac{\rm d}{{\rm d}t} \cEl_{21}x(t)\bigr)\!-\!\cEl_{11}^*\tfrac{\rm d}{{\rm d}t}\cEl_{11} x(t)\\
\overset{\eqref{eq:Edef}}{\underset{\&\eqref{eq:fdaedef}}{=}}&\,\mathcal{E}^*f(t)-\mathcal{E}^*\tfrac{\rm d}{{\rm d}t}\mathcal{E} x(t)\in\partial\varphi(x(t))=\partial\varphi(x_1(t)),\\[1mm]
\cAl_{22}\cAl_{22}^*x_2(t)+\tfrac{\rm d}{{\rm d}t}\cEl_{21} x_1(t)
\overset{\eqref{eq:subsdaeref}}{=}&\,u(t).
\end{aligned}\]
So far, we have proven that $(x_1,x_2)$ fulfills \ref{item:solsysa1})-\ref{item:solsysc}).
Since, by Theorem~\ref{lem:symmdae}, $x$ satisfies
\eqref{eq:DAEBarb1}--\eqref{eq:DAEBarb2}, we obtain from \eqref{eq:Edef} and \eqref{eq:subsdaeref} that
\eqref{eq:DAEBarbsys1}--\eqref{eq:DAEBarbsys3.5} hold. Moreover,
for all $0<t_0 \leq t_1\leq T$,
{\allowdisplaybreaks\begin{align*}
\varphi(x_1(t_1))-&\,\varphi({x_1}(t_0))\overset{\eqref{eq:subsdaeref}}{=}\varphi(x(t_1))-\varphi(x(t_0))\\
\overset{\eqref{eq:DAEBarb5}}{=}
&\,\int_{t_0}^{t_1}\langle f(\tau),\tfrac{\rm d\,}{{\rm d}\tau}\mathcal{E} x(\tau)\rangle_{Y\times U}\,{\rm d}\tau-\int_{t_0}^{t_1}\|\tfrac{\rm d\,}{{\rm d}\tau} \mathcal{E}{x}(\tau)\|^2_{Y\times U}\,{\rm d}\tau\\
\overset{\eqref{eq:Edef}}{\underset{\&\eqref{eq:fdaedef}}{=}}&
\int_{t_0}^{t_1}\bigl\langle\cAl_{22}^{-1}u(\tau),\cAl_{22}^{-1}\tfrac{\rm d\,}{{\rm d}\tau}\cEl_{21}x(\tau)\bigr\rangle_{U} \,{\rm d}\tau\\
&\quad-\int_{t_0}^{t_1}\|\tfrac{\rm d\,}{{\rm d}\tau} \cEl_{11}{x}(\tau)\|^2_{Y}
+\|\cAl_{22}^{-1}\tfrac{\rm d\,}{{\rm d}\tau} \cEl_{21}{x}(\tau)\|^2_{U}\,{\rm d}\tau\\
\overset{\eqref{eq:subsdaeref}}{=}&
\int_{t_0}^{t_1}\bigl\langle \cAl_{22}^{-1}u(\tau), \cAl_{22}^{-1}u(\tau)-\cAl_{22}^*x_2(\tau)\bigr\rangle_{{U}}\,{\rm  d}\tau\\
&\quad-\int_{t_0}^{t_1}\bigl\|\tfrac{\rm d\,}{{\rm d}\tau} \cEl_{11}{x_1}(\tau)\bigr\|^2_{Y}
+\bigl\|\cAl_{22}^{-1}u(\tau)-\cAl_{22}^*x_2(\tau)\bigr\|^2_{U}\,{\rm d}\tau\\
=\enskip & \!\!
\int_{t_0}^{t_1}\bigl\|\cAl_{22}^{-1}u(\tau)\bigr\|^2_{{U}} -\bigl\langle \cAl_{22}^*x_2(\tau),\cAl_{22}^{-1}u(\tau)\bigr\rangle_{U}-\bigl\|\cAl_{22}^*x_2(\tau)\bigr\|^2_{U} \,{\rm d}\tau\\
&\quad -\!\int_{t_0}^{t_1} \bigl\|\tfrac{\rm d\,}{{\rm d}\tau} \cEl_{11}{x_1}(\tau)\|^2_{Y}
+\|\cAl_{22}^{-1}u(\tau)\bigr\|^2_{U}
-\!2\bigl\langle\cAl_{22}^*x_2(\tau),\cAl_{22}^{-1}u(\tau)\bigr\rangle_{U}\,{\rm d}\tau\\
=\;&\!\!\int_{t_0}^{t_1}\bigl\langle x_2(\tau), u(\tau)\bigr\rangle_{U}\,{\rm d}\tau
-\int_{t_0}^{t_1}\bigl\|\cAl_{22}^*x_2(\tau)\bigr\|^2_{U}\, {\rm d}\tau
-\int_{t_0}^{t_1}\bigl\|\tfrac{\rm d\,}{{\rm d}\tau} \cEl_{11}{x_1}(\tau)\bigr\|^2_{Y}\, {\rm d}\tau.
\end{align*}}
Thus, \eqref{eq:DAEBarbsys4} is fulfilled.

If, further, $(z_{1,0},\mathcal{A}_{22}^{-1}z_{2,0})\in \mathcal{E} D(\varphi)$, then Theorem~\ref{lem:symmdae} yields \eqref{eq:DAEBarb3} and \eqref{eq:DAEBarb4}, whence \eqref{eq:subsdaeref} leads to
\eqref{eq:DAEBarbsys5}--\eqref{eq:DAEBarbsys8}. The identity \eqref{eq:DAEBarbsys4} for
all \mbox{$0\leq t_0\leq t_1\leq T$} can be concluded from the corresponding statement in Theorem~\ref{lem:symmdae} and the argumentation of the proof of \eqref{eq:DAEBarbsys4}.\vspace*{-1mm}
\end{proof}

\vspace*{-3mm}
\begin{remark}\ \vspace*{-2mm}
\begin{romanlist}[a)]
\item The additional features of the solution of system \eqref{eq:opdaegrad} in the case where \linebreak $(z_{1,0},\mathcal{A}_{22}^{-1}z_{2,0})\in \mathcal{E} D(\varphi)$ are guaranteed if and only if $(z_{1,0},\mathcal{A}_{22}^{-1}z_{2,0}) = \mathcal{E} x_{{0}}$ for some \mbox{$x_{0}\in D(\varphi)+(\ker\cEl_{11}\cap \ker\cEl_{21})$}.

\item Loosely speaking, Corollary~\ref{thm:infDAEsys} states that $\cEl_{11}x_1$ and $\cEl_{21}x_1$ are differentiable almost everywhere on the open interval $(0,\infty )$, and this for every choice of initial values in the closure of the range of $\mathcal{E}$. This regularization effect occurs in the theory of linear semigroups in the case of differentiable semigroups or, in particular, in the case of analytic semigroups. In other words, in the theory of linear abstract ordinary differential equations of type $\dot{x}(t) + \cAl x(t) = f(t)$, this phenomenon occurs, for example, if the operator $\cAl$ is densely defined and {\em sectorial} \cite[Chapter~II.4]{EngeNage00}. If system \eqref{eq:opdaegrad} is linear (equivalently, $\varphi$ is a~quadratic functional), then by a~careful inspection of the proofs of Lemma~\ref{lem:modfun} and Corollary~\ref{thm:infDAEsys}, it can be seen that the dynamics of \eqref{eq:opdaegrad} is governed by a~nonnegative and self-adjoint linear operator $\cAl$. Such an~operator is sectorial, whence the associated abstract ordinary differential equation has the aforementioned smoothing property.

\item Note that we have not proven the measurability of the solution $(x_1,x_2)$ of \eqref{eq:opdaegrad}, but only measurability of the functions $t\mapsto \tfrac{d}{dt}\cEl_{11} x_1(t)$, \mbox{$t\mapsto \tfrac{d}{dt}\cEl_{21} x_1(t)$}, $t\mapsto x_2(t)$ and, if additionally $\partial\varphi$ is a~function, $t\mapsto\partial\varphi(x(t))$. To prove measurability of $x_2$, some additional assumptions have to be imposed. As compositions of continuous and measurable functions are measurable, such an assumption guaranteeing measurability of $x_1$ can be, for instance, that the mapping $(\cEl_{11} x_1,\cEl_{21} x_1,\partial\varphi(x_1))\mapsto x_1$ is well-defined in some sense and moreover continuous. Such an argument is used in the forthcoming section, where we study the solvability of the coupled MQS system \eqref{eq:MQS}.
\end{romanlist}
\end{remark}

\section{Back to the Coupled MQS System: Existence and Regularity of Solutions}
\label{sec:MQSsolv}

Having developed the framework on abstract differential-algebraic gradient systems, we are now ready to prove the main result of this paper, namely the existence of solutions to the
coupled MQS system  \eqref{eq:MQS}. A~key ingredient is that, by Proposition~\ref{prop:enfun}~\ref{prop:enfun3}), the second summand in the equation~\eqref{eq:MQS1} of the coupled MQS is the subgradient of the magnetic energy $E$ as defined in \eqref{eq:varphiA}.

\begin{theorem}[Existence, uniqueness and regularity of solutions]\label{thm:solMQS}
Let \mbox{$\mathit{\Omega}\subset\mathbb{R}^3$} with subdomains $\mathit{\Omega}_C$ and $\mathit{\Omega}_I$ satisfy Assumption~\textup{\ref{ass:omega}}, and let $X(\mathit{\Omega},\mathit{\Omega}_C)$ and $X_0(\curl,\mathit{\Omega},\mathit{\Omega}_C)$ be defined as in \eqref{eq:statespace} and \eqref{eq:statespace2}, respectively. Further, let the initial and winding functions be as in Assumption~\textup{\ref{ass:init}} and the material parameters as in Assumption~\textup{\ref{ass:material}}.
Let $T>0$ be fixed and $\bm{v}\in L^2([0,T];\mathbb{R}^m)$.
Then the coupled MQS system \eqref{eq:MQS} admits a unique weak solution $(\bm{A},\bm{i})$ on $[0,T]$ in the sense of Definition~\textup{\ref{def:sol}}. This solution has the following properties:
\begin{align}
\left(t\mapsto t^{1/2}\tfrac{\rm d}{{\rm d}t}(\sigma\bm{A}(t))\right)\in&\, L^2([0,T];X(\mathit{\Omega},\mathit{\Omega}_C)),\label{eq:MQSDAEsys2}\\
\left(t\mapsto t^{1/2}\tfrac{\rm d}{{\rm d}t}\int_\mathit{\Omega} \chi^\top\bm{A}(t) \, {\rm d}\xi\right)\in&\, L^2([0,T];\R^m),\label{eq:MQSDAEsys3}\\
\left(t\mapsto t^{1/2}\;(\nabla\times\bm{A}(t))\right)\in&\, L^\infty([0,T];X(\mathit{\Omega},\mathit{\Omega}_C)),\label{eq:MQSDAEsys2.5}\\
\left(t\mapsto t^{1/2}\, \bm{i}(t)\right)\in&\, L^2([0,T];\R^m),\label{eq:MQSDAEsys3.5}\\
\left(t\mapsto t^{1/2}\, \nu(\cdot,\|\nabla \times \bm{A}(t)\|_2)
\nabla \times \bm{A}(t)\right)\in&\,L^2([0,T];H(\curl,\mathit{\Omega})).\label{eq:MQSDAEsys4}
\end{align}
For almost all $t\in[0,T]$,
\begin{align}
\tfrac{\partial}{\partial t}\left(\sigma\bm{A}(t)\right) + \nabla \times \left(\nu(\cdot,\|\nabla \times \bm{A}(t)\|_2)
\nabla \times \bm{A}(t)\right)  & =  \, \chi\, \bm{i}(t) ,  \label{eq:classsolMQS} \\
\tfrac{\rm d }{{\rm d} t}\int_\mathit{\Omega} \chi^\top\bm{A}(t)\, {\rm d}\xi + R\,\bm{i}(t) & =\, \bm{v}(t) .\label{eq:classsolMQS2}
\end{align}
If, moreover, $\bm{A}_0\in X_0(\curl,\mathit{\Omega},\mathit{\Omega}_C)$, then the solution fulfills
\begin{align}
\nabla\times\bm{A}\in&\, L^\infty([0,T];X(\mathit{\Omega},\mathit{\Omega}_C)),\label{eq:MQSDAEsys0.5s}\\
\sigma\bm{A}\in&\, H^1([0,T];X(\mathit{\Omega},\mathit{\Omega}_C)),\label{eq:MQSDAEsys2s}\\
\int_\mathit{\Omega} \chi^\top\bm{A} \, {\rm d}\xi\in&\, H^1([0,T];\R^m),\label{eq:MQSDAEsys3s}\\
\bm{i}\in&\, L^2([0,T];\R^m).\label{eq:MQSDAEsys3.5s}
\end{align}
\end{theorem}

\begin{proof}
Step 1: First, we verify that, by taking the spaces $X=Y=X(\mathit{\Omega},\mathit{\Omega}_C)$ and $U=\R^m$, the operators $\cEl_{11} : X(\mathit{\Omega},\mathit{\Omega}_C) \to X(\mathit{\Omega},\mathit{\Omega}_C)$, \mbox{$\cEl_{12} : X(\mathit{\Omega},\mathit{\Omega}_C) \to \R^m$} and $\cAl_{22} : \R^m \to\R^m$ defined in \eqref{eq:EAiiop}
and the functional $\varphi= E$ with the magnetic energy~$E$ as in \eqref{eq:varphiA} fulfill the assumptions of Corollary~\ref{thm:infDAEsys}.
It follows from Proposition~\ref{prop:enfun} that $ E:X\to\mathbb{R}_{\ge0}\cup\{\infty\}$ is densely defined, convex, lower semicontinuous and $\mathcal{E}$-elliptic for  $\mathcal{E}\in\mathcal{L}(X(\mathit{\Omega},\mathit{\Omega}_C),X(\mathit{\Omega},\mathit{\Omega}_C)\times \mathbb{R}^m)$  as in \eqref{eq:Edef}.

Step~2: By using the representation of $\partial E$ from Proposition~\ref{prop:enfun}~\ref{prop:enfun3}), we conclude from Corollary~\ref{thm:infDAEsys} that there exists a function $(\bm{A},\bm{i}):[0,T]\to X(\mathit{\Omega},\mathit{\Omega}_C)\times\R^m$ with the following properties, by respectively referring to \ref{item:solsysa1})-\ref{item:solsysc}) in Corollary~\ref{thm:infDAEsys}:\vspace*{-1mm}
\begin{romanlist}[a)]
\item\label{item:solproof1}
$\sigma \bm{A} \in C([0,T] ; X(\mathit{\Omega},\mathit{\Omega}_C)) \cap H^1_{\rm loc} ((0,T] ; X(\mathit{\Omega},\mathit{\Omega}_C))$ and $\sigma \bm{A}(0)=\sigma \bm{A}_0$;
\item\label{item:solproof1a}
$\int_\mathit{\Omega} \chi^\top\bm{A} \, {\rm d}\xi \in C ([0,T]; \R^m) \cap H^1_{\rm loc} ((0,T] ; \R^m )$ and $\int_\mathit{\Omega} \chi^\top\bm{A}(0) \, {\rm d}\xi=\int_\mathit{\Omega} \chi^\top\bm{A}_0 \, {\rm d}\xi$;
\item\label{item:solproof5} $\bm{i}\in L^{2}_{\rm loc}((0,T];\R^m)$;
\item\label{item:solproof6} $\nu(\cdot,\|\nabla \times \bm{A}(t)\|_2)
\nabla \times \bm{A}(t)\in H(\curl,\mathit{\Omega})$,
and the equations \eqref{eq:classsolMQS} and \eqref{eq:classsolMQS2} hold for almost all $t\in[0,T]$.
\end{romanlist}

\vspace*{-2mm}
Step~3: We show that $(\bm{A},\bm{i})$ is a weak solution of the coupled MQS system \eqref{eq:MQS} in the sense of Definition~\ref{def:sol}. By using the results from Step~2, it remains to prove that\vspace*{-2mm}
\begin{romanlist}[(i)]
\item $\bm{A}\in L^2([0,T],X_0(\curl,\mathit{\Omega},\mathit{\Omega}_C))$, and
\item for all $\bm{F}\in X_0(\curl,\mathit{\Omega},\mathit{\Omega}_C)$, equations \eqref{eq:weak} are fulfilled for almost all $t\in[0,T]$.
\end{romanlist}
Statement (ii) is a~simple consequence of \ref{item:solproof6}), that is, equations \eqref{eq:classsolMQS} and \eqref{eq:classsolMQS2}, and the integration by parts formula \eqref{eq:curladj}.

In order to prove (i), first note that by the properties \ref{item:solproof1}) and \ref{item:solproof5}) above,
\[
 \tfrac{{\rm d}}{{\rm d} t}\left(\sigma\bm{A}\right) , \, \chi\, \bm{i} \in L^2_{\rm loc} ((0,T] ; X(\mathit{\Omega},\mathit{\Omega}_C) ) .
\]
Hence, by property \ref{item:solproof6}) (more precisely, by equation \eqref{eq:classsolMQS}),
\begin{equation} \label{eq:new1}
 \partial E(\bm{A} ) = \nabla \times \left(\nu(\cdot,\|\nabla \times \bm{A}\|_2)
\nabla \times \bm{A}\right) \in L^2_{\rm loc} ((0,T] ; X(\mathit{\Omega},\mathit{\Omega}_C) ) .
\end{equation}
Second, let the operator $\cAl_{11}$ be defined as in \eqref{eq:A11}. Note that $\cAl_{11} = \partial E$ on $D(\partial E)$. By the estimates \eqref{eq:curlest} and \eqref{eq:A11diss} from Lemmas \ref{lem:denscoerc} and \ref{lem:A11mon}, respectively, for all \mbox{$\bm{A}_1$, $\bm{A}_2\in X_0(\curl,\mathit{\Omega},\mathit{\Omega}_C)$},
\begin{align*}
 \| \bm{A}_1 - \bm{A}_2&\|_{H(\curl , \mathit{\Omega})}^2 = \,\| \bm{A}_1 - \bm{A}_2\|_{L^2 (\mathit{\Omega} ;\R^3)}^2 + \| \nabla\times (\bm{A}_1 - \bm{A}_2)\|_{L^2 (\mathit{\Omega} ;\R^3)}^2 \\
 &\leq\, \tfrac{L_C}{\sigma_C^2} \, \| \sigma (\bm{A}_1 - \bm{A}_2)\|_{L^2 (\mathit{\Omega};\R^3)}^2 + (L_C +1) \| \nabla\times (\bm{A}_1 - \bm{A}_2\|_{L^2 (\mathit{\Omega};\R^3)}^2 \\
 &\leq \,\tfrac{L_C}{\sigma_C^2} \, \| \sigma (\bm{A}_1 -\bm{A}_2)\|_{L^2 (\mathit{\Omega};\R^3)}^2 + \tfrac{L_C+1}{m_\nu} \, \langle \bm{A}_1 - \bm{A}_2 , \cAl_{11}(\bm{A}_1) - \cAl_{11}(\bm{A}_2)\rangle \\
 &\leq \,\tfrac{L_C}{\sigma_C^2} \, \| \sigma (\bm{A}_1 - \bm{A}_2)\|_{L^2 (\mathit{\Omega};\R^3)}^2 \\
  &\,\phantom{\leq + } + \tfrac{L_C+1}{m_\nu} \, \| \bm{A}_1 - \bm{A}_2 \|_{H(\curl , \mathit{\Omega})} \| \cAl_{11}(\bm{A}_1) - \cAl_{11}(\bm{A}_2) \|_{X_0(\curl,\mathit{\Omega},\mathit{\Omega}_C)'}.
\end{align*}
This inequality combined with Young's inequality implies that there is a constant $C\geq 0$ such that, for all $\bm{A}_1$, $\bm{A}_2\in X_0(\curl,\mathit{\Omega},\mathit{\Omega}_C)$,
\begin{equation}
 \| \bm{A}_1 - \bm{A}_2\|_{H(\curl , \mathit{\Omega})}^2 \leq C \bigl(\| \sigma(\bm{A}_1 -\bm{A}_2)\|_{L^2 (\mathit{\Omega};\R^3)}^2 + \| \cAl_{11}(\bm{A}_1) -\cAl_{11}(\bm{A}_2) \|_{X_0(\curl,\mathit{\Omega},\mathit{\Omega}_C)'}^2 \bigr).
\end{equation}
In other words, the mapping
$$
\begin{array}{rcl}
 X_0(\curl,\mathit{\Omega},\mathit{\Omega}_C) & \to &X(\mathit{\Omega},\mathit{\Omega}_C) \times X_0(\curl,\mathit{\Omega},\mathit{\Omega}_C)' , \\
 \bm{A} & \mapsto& (\sigma \bm{A} , \cAl_{11}(\bm{A}) )
\end{array}
$$
has a Lipschitz continuous inverse defined on the range of the above mapping.  Thus, the continuity of $\sigma\bm{A}$ with values in $X(\mathit{\Omega},\mathit{\Omega}_C)$ and \eqref{eq:new1} imply
\[
 \bm{A} \in L^2_{\rm loc} ((0,T] ; X_0(\curl,\mathit{\Omega},\mathit{\Omega}_C)) .
\]
However, since the mapping $t\mapsto  E(\bm{A}(t))$ is integrable on $[0,T]$ by Corollary~\ref{thm:infDAEsys}, and by the estimate \eqref{eq:estE}, we have $\nabla \times \bm{A}\in L^2([0,T];X(\mathit{\Omega},\mathit{\Omega}_C))$. Thus, the continuity of~$\sigma\bm{A}$ with values in $X(\mathit{\Omega},\mathit{\Omega}_C)$ and Lemma \ref{lem:denscoerc} actually imply the stronger statement
\[
 \bm{A} \in L^2 ([0,T] ; X_0(\curl,\mathit{\Omega},\mathit{\Omega}_C)) ,
\]
which is property (i) above.

Step~4: Next, we prove \eqref{eq:MQSDAEsys2}-\eqref{eq:MQSDAEsys4}.
By using Step~1, \eqref{eq:MQSDAEsys2}, \eqref{eq:MQSDAEsys3} and \eqref{eq:MQSDAEsys3.5} are, respectively, consequences of \eqref{eq:DAEBarbsys2}, \eqref{eq:DAEBarbsys3} and \eqref{eq:DAEBarbsys3.5} in Corollary~\ref{thm:infDAEsys}. Further, by invoking \eqref{eq:estE} again, we see that \eqref{eq:DAEBarbsys1.5} implies \eqref{eq:MQSDAEsys2.5}. The remaining relation \eqref{eq:MQSDAEsys4} can be verified as follows.
Using \eqref{eq:Mnugreater}, we obtain that $\nu(\cdot,\|\nabla \times \bm{A}\|_2)$ is essentially bounded. This together with \eqref{eq:MQSDAEsys2.5} yields that $\left(t\mapsto t^{1/2}\, \nu(\cdot,\|\nabla \times \bm{A}(t)\|_2)
\nabla \times \bm{A}(t)\right)\in L^2([0,T];L^2(\mathit{\Omega};\mathbb{R}^3))$. Since, moreover, by \eqref{eq:classsolMQS}, we have
\[
\nabla \times \left(\nu(\cdot,\|\nabla \times \bm{A}(t)\|_2)
\nabla \times \bm{A}(t)\right)  =  \, -\tfrac{\rm d}{{\rm d} t}\left(\sigma\bm{A}(t)\right)+\chi \bm{i}(t),
\]
the relations \eqref{eq:MQSDAEsys2} and \eqref{eq:MQSDAEsys3.5} lead to
\[\left(t\mapsto t^{1/2}\, \nabla \times \nu(\cdot,\|\nabla \times \bm{A}(t)\|_2)
\nabla \times \bm{A}(t)\right)\in L^2([0,T];L^2(\mathit{\Omega};\mathbb{R}^3)),\] whence \eqref{eq:MQSDAEsys4} holds.

Step~5: Finally, we show that \eqref{eq:MQSDAEsys0.5s}-\eqref{eq:MQSDAEsys3.5s} hold, if the initial value additionally fulfills $\bm{A}_0\in X_0(\curl,\mathit{\Omega},\mathit{\Omega}_C) = D(E)$.
The statements \eqref{eq:MQSDAEsys2s}-\eqref{eq:MQSDAEsys3.5s} can be proven analogously to the results in Step~4 by invoking \eqref{eq:DAEBarbsys6}-\eqref{eq:DAEBarbsys8} in Corollary~\ref{thm:infDAEsys}. To prove \eqref{eq:MQSDAEsys0.5s}, we first make use of
Corollary~\ref{thm:infDAEsys} which implies, via \eqref{eq:DAEBarbsys5}, that $\left(t\mapsto  E(\bm{A}(t))\right)\in W^{1,1}([0,T])$ and, hence, {$\left(t\mapsto  E(\bm{A}(t))\right)\in L^{\infty}([0,T])$.} This, together with \eqref{eq:estE}, yields that $\|\nabla\times\bm{A}(t)\|_{L^2(\mathit{\Omega};\mathbb{R}^3)}$ is essentially bounded.
\end{proof}

\begin{remark}
As mentioned in the introduction, linear coupled MQS systems with~$\nu$ being independent of $\bm{A}$ have been studied in \textup{\cite{NicT14}}, where it has additionally been assumed that the non-conducting subdomain $\Omega_I$ is connected. In the language of Assumption~\ref{ass:omega}, this means that $q=0$ and $\mathit{\Omega}_I=\mathit{\Omega}_{I,{\rm ext}}$. Further, it has been seeked for solutions in which the magnetic vector potential evolves in the space
\[
Y(\mathit{\Omega})=\bigl\{\bm{F}\in H_0(\curl,\mathit{\Omega})\cap L_2(\divg\!=\!0,\mathit{\Omega}_C\cup\mathit{\Omega}_I;\R^3) \;:\;  \langle \left.\bm{F}\right|_{\Omega_I}\cdot\bm{n},1\rangle_{L^2(\Gamma_{{\rm ext}})}=0\bigr\}.
\]
Hereby, \mbox{$\left.\bm{F}\right|_{\Omega_I}\!\cdot\bm{n}$} stands for the normal boundary trace of the restriction of $\bm{F}$ to the non-conducting domain (this normal boundary trace is well-defined by \textup{\cite[Theo\-rem~I.2.5]{GiraRavi86}} and the fact that $\bm{F}$ has a~weak divergence $\nabla\cdot \bm{F}\in L^2(\Omega)$). Note that the space $Y(\mathit{\Omega})$ coincides with our space $X_0(\curl,\mathit{\Omega},\mathit{\Omega}_C)$ under the assumption in \textup{\cite{NicT14}} that the non-conducting subdomain is connected. To see that $Y(\mathit{\Omega})\subset X_0(\curl,\mathit{\Omega},\mathit{\Omega}_C)$, let $\psi\in H^1_0(\Omega)$ with $\left.\psi\right|_{\Gamma_{{\rm ext}}}=c_{{\rm ext}}$ for some $c_{{\rm ext}}\in\R$. Then, by using integration by parts with the weak divergence, we obtain for all $\bm{F}\in Y(\mathit{\Omega})$ that
\[\begin{aligned}
\langle\nabla\psi,\bm{F}\rangle_{L^2(\Omega;\R^3)}
&=-\langle\psi,\underbrace{\nabla\cdot\bm{F}}_{=0}\rangle_{L^2(\Omega)}\!+\!
\langle\left.\bm{F}\right|_{\Omega_I}\!\cdot\bm{n},\underbrace{\!\psi\!}_{=c_{{\rm ext}}}\rangle_{L^2(\Gamma_{{\rm ext}})}\!+\!
\langle\left.\bm{F}\right|_{\Omega_I}\!\cdot\bm{n},\underbrace{\!\psi\!}_{=0}\rangle_{L^2(\partial\Omega)}\\
&=c_{{\rm ext}}\langle\left.\bm{F}\right|_{\Omega_I}\cdot\bm{n},1\rangle_{L^2(\Gamma_{{\rm ext}})}=0.
\end{aligned}\]
Hence, $\bm{F}\in X(\mathit{\Omega},\mathit{\Omega}_C)$. Further, $Y(\mathit{\Omega})\subset H_0(\curl,\mathit{\Omega})$ implies
\mbox{$\bm{F}\in X_0(\curl,\mathit{\Omega},\mathit{\Omega}_C)$}.
On the other hand, if $\bm{F}\in X_0(\curl,\mathit{\Omega},\mathit{\Omega}_C)$, then $\bm{F}\in L^2(\divg\!=\!0, \mathit{\Omega}_C \cup \mathit{\Omega}_I ;\R^3)$ by \eqref{eq:divfree}. To prove that $\bm{F}\in Y(\mathit{\Omega})$, it remains to show that the integral of the normal trace of $\left.\bm{F}\right|_{\Omega_I}$ over $\Gamma_{{\rm ext}}$ vanishes. To see this, let $\psi\in H^1_0(\Omega)$ be  such that $\left.\psi\right|_{\Omega_C}\equiv1$ (which exists by $\overline{\mathit{\Omega}}_{C}\subseteq {\mathit{\Omega}}$). Then $\nabla\psi\in G(\Omega,\Omega_C)$, and, by further invoking that $\nabla\psi$ vanishes on $\Omega_C$, we obtain
\[\begin{aligned}
0=&\,\langle\nabla\psi,\bm{F}\rangle_{L^2(\Omega;\R^3)}=\,\langle\nabla\psi,\bm{F}\rangle_{L^2(\Omega_I;\R^3)}\\
=&\,-\langle\psi,\underbrace{\nabla\cdot\bm{F}}_{=0}\rangle_{L^2(\Omega_I)}+
\langle\left.\bm{F}\right|_{\Omega_I}\cdot\bm{n},\underbrace{\psi}_{=1}\rangle_{L^2(\Gamma_{{\rm ext}})}+
\langle\left.\bm{F}\right|_{\Omega_I}\cdot\bm{n},\underbrace{\psi}_{=0}\rangle_{L^2(\partial\Omega)}\\
=&\,\langle\left.\bm{F}\right|_{\Omega_I}\cdot\bm{n},1\rangle_{L^2(\Gamma_{{\rm ext}})}.
\end{aligned}\]
Existence of a solution $\bm{A}\in L^2([0,T];Y(\mathit{\Omega}))$ with $\sigma\bm{A}\in W^{1,1}([0,T];Y(\mathit{\Omega})')$ is shown in
\textup{\cite[Corollary~3.13]{NicT14}}.
 For the case where the voltage and initial value additionally fulfill $\bm{v}\in H^1([0,T];\R^m)$ and $\bm{A}_0\in H_0(\curl,\Omega)$ with $\nu \, \nabla\times \bm{A}_0\in H(\curl,\mathit{\Omega})$, it is proven in \textup{\cite[Theorem~3.11]{NicT14}} that $\bm{A}\in H^1([0,T];H_0(\curl,\mathit{\Omega}))$.
\end{remark}

\section{Conclusion}
We have considered a~quasilinear MQS approximation of Maxwell's equations, which is furthermore coupled with
an~integral equation. By employing the magnetic energy, this system can be reformulated as an abstract differential-algebraic equation involving a~subgradient. For this class of equations, we have developed novel well-posedness and regularity results which we have then applied to the coupled MQS system.

%

%

\end{document}